\documentclass{birkjour}
\usepackage[T1]{fontenc}
\usepackage[utf8]{inputenc}
\usepackage[english]{babel}
\usepackage{latexsym}
\usepackage{url}
\usepackage{amsmath}
\usepackage{amssymb}
\usepackage{amsfonts}
\usepackage{graphicx}
\usepackage{float}
\usepackage{subcaption}
\usepackage{upgreek}
\usepackage{ mathrsfs }
\usepackage{ dsfont }
\usepackage{amsthm}
\usepackage{graphicx}
\usepackage{caption}
\usepackage{booktabs}
\usepackage{color}
\usepackage{framed}
\usepackage{verbatim}
\usepackage{pdflscape,rotating,slashbox,arydshln}
\usepackage[unicode=true,psdextra,
 bookmarks=true,bookmarksnumbered=true,bookmarksopen=false,
 breaklinks=true,pdfborder={0 0 1},backref=false,colorlinks=true]
 {hyperref}
 \newcommand\pdfmath[1]{\texorpdfstring{$#1$}{#1}}
\hypersetup{pdftitle={H-Functions, non-Gaussian measures, Appell systems},
 pdfauthor={Lorenzo Cristofaro, Jos\'e. L.~da Silva, L.~Beghin},
 pdfsubject={H-Functions, non-Gaussian measures, Appell systems},
 pdfkeywords={H-Functions, non-Gaussian measures, Appell systems},
 linkcolor=black,citecolor=black,urlcolor=black,filecolor=black}
\definecolor{shadecolor}{RGB}{248,248,248}

\theoremstyle{plain}
\newtheorem{theorem}{Theorem}[section]
\theoremstyle{definition}
\newtheorem{definition}[theorem]{Definition}

\newtheorem{remark}[theorem]{Remark}
\newtheorem{corollary}[theorem]{Corollary}
\newtheorem{proposition}[theorem]{Proposition}
\newtheorem{lemma}[theorem]{Lemma}
\newtheorem{assumption}[theorem]{Assumption}
\newcommand{\N}{\mathbb{N}}
\newcommand{\Z}{\mathbb{Z}}
\newcommand{\R}{\mathbb{R}}
\newcommand{\C}{\mathbb{C}}
\newcommand{\e}{\mathrm{e}}

\begin{document}
\title{Generalized Wright Analysis in Infinite Dimensions}

\author[L.~Beghin]{Luisa Beghin}

\address{%
Sapienza University of Rome\\
 P.le Aldo Moro, 5, Rome, Italy,}

\email{luisa.beghin@uniroma1.it}

\author[L.~Cristofaro]{Lorenzo Cristofaro}

\address{%
 University of Luxembourg\\
 Esch-sur-Alzette, Luxembourg,
}

\email{lorenzo.cristofaro@uni.lu}

\author[J.~L.~da Silva]{Jos{é} L.~da Silva}

\address{%
 CIMA, Faculty of Exact Sciences and Engineering,\\
  University of Madeira, Campus da Penteada,\\
 9020-105 Funchal, Portugal.}

\email{joses@staff.uma.pt}



\date{January 4, 2025}

\begin{abstract}
This paper investigates a broad class of non-Gaussian measures, $ \mu_\Psi$, associated with a family of generalized Wright functions, $_m\Psi_q$. 
First, we study these measures in Euclidean spaces $\mathbb{R}^d$, then define them in an abstract nuclear triple $\mathcal{N}\subset\mathcal{H}\subset\mathcal{N}'$. We study analyticity, invariance properties, and ergodicity under a particular group of automorphisms. Then we show the existence of an Appell system which allows the extension of the non-Gaussian Hilbert space $L^2(\mu_\Psi)$ to the nuclear triple consisting of test functions' and distributions' spaces,  $(\mathcal{N})^{1}\subset L^2(\mu_\Psi)\subset(\mathcal{N})_{\mu_\Psi}^{-1}$.  Furthermore, thanks to the definition of two transformations, $S_{\mu_{\Psi}}$ and $T_{\mu_{\Psi}}$, we study Donsker's delta as an element within $(\mathcal{N})_{\mu_\Psi}^{-1}$ applying the integral equations fulfilled by $_m\Psi_q$.
\end{abstract}	


	\maketitle

 \section{Introduction}

In the twentieth century, stochastic analysis and measure theory in function spaces were extensively studied to develop models with random phenomena. The classic example is Gaussian analysis, which played a central role in its advancement because of the power of its tools. Two important reference points are white noise theory and the Malliavin calculus; see \cite{HKPS93} and \cite{Nua06}.
These two theories provided analytical tools such as derivatives and integral transforms in function spaces. These tools help us to create infinite dimensional spaces that allow the study of functionals and distributions related to stochastic processes. 
In white noise analysis, many important functionals, such as Feynman-Kac formulae, and distributions like local times, Donsker's delta, and noise of standard or fractional Brownian motion, were properly defined and widely studied.

More specifically, white noise is the measure space $(\mathcal{S}(\R)', \mathcal{C}_\sigma(\mathcal{S}(\R)'), \mu)$, where $\mathcal{S}(\R)'$ is the dual space of the Schwartz space of rapidly decreasing smooth functions,  $\mathcal{C}_\sigma(\mathcal{S}(\R)')$ is the $\sigma$-algebra generated by its weak topology and $\mu$ is the Gaussian measure. Other probability measures can be considered on $\mathcal{S}(\R)'$ or on an abstract nuclear triple $\mathcal{N}\subset \mathcal{H}\subset\mathcal{N}'$, to cover measures for point processes or more general non-Gaussian measures. We refer to the following books for Gaussian analysis \cite{HKPS93,BK95,Oba,RS72} and \cite{Ito88,AKR97a,KDSSU98} for Poisson and related measures. Given this, a general infinite dimensional non-Gaussian analysis was developed to study random variables in non-Gaussian Hilbert spaces, their functionals, and spaces of distributions; see \cite{Dal91, KSWY98}. More recently, an approach based on topological algebras and free probability to study the orthogonal decomposition of non-Gaussian measures and generalized stochastic processes was studied in \cite{ACK22} and \cite{AJS14}, respectively.

In the approach to infinite dimensional non-Gaussian analysis given by \cite{ADKS96,KSWY98}, the Wick-ordered polynomials (in Gaussian analysis) are replaced by a biorthogonal system of polynomials called the Appell systems. Therefore, the constructions and concepts from Gaussian analysis become suitable for a
wide class of measures, even if no Wick-ordered polynomial exists. To prove the existence of Appell systems, the measure must satisfy two properties: its Laplace transform must be a holomorphic functional and satisfy a non-degeneracy or positivity condition, that is, it is strictly positive on nonempty open sets. Once the Appell systems are established, the powerful techniques of Gaussian analysis can be transferred to non-Gaussian analysis to define spaces of test functions, distributions, and integral transforms. The integral transforms called $S$ and $T$-transforms are used to characterize the distributions in terms of spaces of holomorphic functions on locally convex spaces.

In this regard, \cite{Sch88} and \cite{Sch90} introduced a non-Gaussian family of probability measures called grey noise measures, whose characteristic functionals are given by Mittag-Leffler functions. Grothaus et al.~\cite{JahnI} extended the family of grey noise measures to
Mittag-Leffler measures on an abstract nuclear triple $\mathcal{N}\subset\mathcal{H}\subset\mathcal{N}'$. This family of measures $\mu_\beta$, $0<\beta\le1$, on $(\mathcal{N}', \mathcal{C}_{\sigma}(\mathcal{N}'))$ satisfies 
\[ \int_{\mathcal{N}'} \e^{\mathrm{i} \langle \omega, \xi \rangle} \mathrm{d}\mu_\beta(\omega)=E_{\beta}\left( -\frac{1}{2}\langle \xi, \xi \rangle   \right), \qquad \xi \in \mathcal{N},  
\]
where the Mittag-Leffler function $E_{\beta}(x)=\sum_{j = 0}^\infty \frac{1}{\Gamma(\beta j +1)}x^j$, $x \in \R$. Recently, gamma-gray noise, a non-Gaussian setting related to the incomplete gamma function $\Gamma(\rho, \theta +x)$, $x\geq 0$, was also defined; see \cite{BCG23} and \cite{BC23}. The Mittag-Leffler and gamma-grey measures are both generalizations of the Gaussian measure.  
Furthermore, these measures satisfy the necessary conditions for the existence of Appell systems. Therefore, this allows us to introduce the test function and distribution spaces following the methodology described in \cite{KSWY98}. 
The latter allows us to study Donsker's delta as an element in a proper distribution space; see \cite{BC23}.

Another goal given by non-Gaussian analysis is its application to differential equations and stochastic processes. In \cite{JahnII}, specific non-Gaussian distribution spaces are shown to be required to properly define the solution of the fractional heat equation and represent it using a Feynman-Kac formula. Furthermore, the interplay between anomalous diffusion and fractional diffusion equations carries numerous scientific applications, like relaxation type differential equations (with a proper notion of stochastic integral in this framework), continuous time random walks, or viscoelasticity; see \cite{Koc11,MS19,Mai22}. 

The key role of the Mittag-Leffler function in the above infinite dimensional setting is replaced here, for the first time, by the use of generalized Wright functions $_m\Psi_p(\cdot)$, gWf hereinafter, the Mittag-Leffler function being a special case.
Indeed, in this paper, we focus on a new class of non-Gaussian measures $\mu_\Psi$ characterized by gWfs; see Definition~\ref{def:genWrigfuncti} below. In Euclidean space $\mathbb{R}^d$, these measures are absolutely continuous with respect to the Lebesgue measure, and their Radon-Nikodym derivative is given in terms of Fox-H functions. 
Due to their great flexibility, the class of Fox-H functions (first defined in \cite{Fox61}) may cover many known elementary and special functions and is used to solve problems in different branches of science. They have been used to express solutions to differential equations through series expansions and Mellin-Barnes integrals. Bessel functions, hypergeometric, or confluent hypergeometric functions are some of the most famous examples.  Indeed, due to their analytical properties or expression through continuous fractions, the latter have gained extensive theoretical utility and have inspired advances in diverse mathematical domains; in particular to integral equations, Feynman-Kac formula (see \cite{Gri13}, \cite{KR09},  \cite{CPVW08}, \cite{KK25}) which in this framework is postponed for further work.

Other applications include statistics, astrophysics, diffusion problems, probability, and operator theory; see \cite{MS78, Spri79, Kir94, Mai22}.

During the past thirty years, numerous authors have established and examined the correlation between fractional calculus, special functions, and probability. In particular, the interplay between some transcendental functions of Wright type and time-fractional diffusion processes was thoroughly investigated; see, for example, \cite{MPS05}, \cite{MMP10}, \cite{GLM99}, and references therein.
Mittag-Leffler functions, gWfs, Le Roy functions, and many others are strictly connected to fractional operators, as shown, for example, in \cite{KST02,GKMR20,GG18,DLS21}. 

Now we describe our results in more detail. First, we would like to understand the class of probability measures in which we are interested in the Euclidean space $\mathbb{R}^d$, $d\in\mathbb{N}$. This class is denoted by $\mathcal{M}_\Psi(\mathbb{R}^d)$. To be precise, we show that each measure $\mu_\Psi\in\mathcal{M}_\Psi(\mathbb{R}^d)$ is represented as a mixture of Gaussian measures and identify the class of mixing measures on $(0,\infty)$. Each element $\mu_\Psi\in\mathcal{M}_\Psi(\mathbb{R}^d)$ has two fundamental properties:
\begin{itemize}
    \item Its mixed moments of all orders are finite;
    \item It is absolutely continuous with respect to the Lebesgue measure.
\end{itemize}
In addition, we introduce a sequence of monic polynomials $F_n^H(\cdot)$ for $d=1$. These are given in terms of Hermite polynomials with parameter $\sqrt{\tau}$, where $\tau>0$; see Equation~\eqref{eq:FoxHermitePolynomials}.
An abstract nuclear triple is used to generalize the results to infinite dimensions for later applications. Hence, we introduce the class of probability measures $\mathcal{M}_\Psi(\mathcal{N}')$ by providing its characteristic functionals and using Minlos' theorem. The finite-dimensional projections of these measures coincide with the class $\mathcal{M}_\Psi(\mathbb{R}^d)$. We show that each element $\mu_\Psi\in\mathcal{M}_\Psi(\mathcal{N}')$ has mixed moments of all orders finite; see Theorem~\ref{thm:MomentmuHInfiniteDimension}. 
To introduce the test and generalized function spaces, we need to ensure that the class $\mathcal{M}_\Psi(\mathcal{N}')$ satisfies the properties \eqref{PropertyA1} and \eqref{PropertyA2}, that is, each measure is analytic and positive on nonempty sets. These results are proved in Theorem~\ref{thm:holomorphic-LT} and Theorem~\ref{thm:mixture}, respectively. The properties \eqref{PropertyA1} and \eqref{PropertyA2} are necessary and sufficient to introduce test functions and distribution spaces as well their characterization via integral transforms; see Subsection~\ref{subsec:charact-thms}.     

Thus, the content of Sections~\ref{sec:FinidimeFoxHmeasFHPolyno}--\ref{sec:AnalypropFHmeasu} has been described. It remains to add that in Section~\ref{sec:DonskDelta} we give an application in this non-Gaussian setting, namely we define the Donsker delta as a well-defined object in a certain distribution space.

\section{Finite Dimensional Generalized Wright Measures}\label{sec:FinidimeFoxHmeasFHPolyno}

In this section, we define the class $\mathcal{M}_\Psi(\R^d)$ of generalized Wright measures (for short gWm) on $\R^d$, $d \in \N$. More precisely, this class of probability measures on $\R^d$ has the Fourier transform given by a gWf $_m\!\Psi_p(\cdot)$; see Definition~\ref{def:genWrigfuncti} below for details. We show that every measure $\mu_\Psi^d$ of this class is a mixture of Gaussian measures with a certain probability measure on $(0,\infty)$, expressed by Fox-$H$ functions; cf.~Corollary~\ref{cor:mixture-fd}. We also compute their mixed moments and densities, denoted by $\varrho_H^d(\cdot)$. We conclude this section by defining the Fox-Hermite polynomials $F_n^H(x)$, $x \in \R$, and some of their properties.

\subsection{Fox-\pdfmath{H} Densities and Generalized Wright Functions }

To introduce the class of gWm (see Definition~\ref{def:finite-dim-gWm} below), we recall the assumptions on the parameters $a_i,b_j,\alpha_i,\beta_j, i=1,\dots,p, j=1,\dots,m$ involved in $_m \Psi_q(\cdot)$. The constraints on the parameters in Assumption~\ref{ass:AllHdensity}, together with $a^* \in (0,1)$, ensure the analyticity of the gWfs on the whole complex plane. Indeed, in Assumption~\ref{ass:AllHdensity} we state the conditions under which the Fox-$H$ densities $\varrho(\cdot)$ have all moments. These densities, called F$H$dam, serve the purpose of defining gWfs, as detailed in Lemma~\ref{lem:HFiniteDensities}. 
Fox-$H$ functions are defined through a Mellin-Barnes integral as follows.

\begin{definition}[Fox-$H$ functions; see \cite{SaiKil}]
Let $m,n,p,q \in \mathbb{N}_0$ be given with $0 \leq m \leq q$, $0 \leq n \leq p$. Let $a_i, b_j \in \mathbb{C}$ and $\alpha_i,\beta_j>0$ for $i=1,..,p, j=1,..,q$. The Fox-$H$ function is defined through the Mellin-Barnes integral as:\\
\begin{equation} \label{def:KernelFunctionFoxH}  
H_{p,q}^{m,n}\left[z \Bigg| \genfrac{}{}{0pt}{}{(a_i,\alpha_i)_{1,p}}{(b_j,\beta_j)_{1,q}}\right]:=\frac{1}{2\pi \mathrm{i}}\int_{\mathcal{L}} \mathcal{H}^{m,n}_{p,q}(s)z^{-s}\mathrm{d}s, \quad z \in \mathbb{C},  
\end{equation}
where 
\[
\mathcal{H}^{m,n}_{p,q}(s)= \frac{\prod_{j=1}^{m} \Gamma({b_j+s\beta_j}) \prod_{i=1}^{n}\Gamma({1-a_i-s\alpha_i})}{ \prod_{i=n+1}^{p}\Gamma({a_i+s\alpha_i})\prod_{j=m+1}^{q} \Gamma({1-b_j
				-s\beta_j})}
\]
and $\mathcal{L}$ is a contour that separates the poles of $\mathcal{H}^{m,n}_{p,q}(s)$ given by $\prod_{j=1}^{m} \Gamma({b_j+s\beta_j})$ from those of $ \prod_{i=1}^{n}\Gamma({1-a_i-s\alpha_i})$.
\end{definition}
\begin{remark}
    The contour $\mathcal{L}$ can take three distinct shapes, which we refer to as $\mathcal{L}_{-\infty}$, $\mathcal{L}_{+\infty}$, and $\mathcal{L}_{\mathrm{i}\gamma\infty}$, defined as follows:
\begin{itemize}
	\item $\mathcal{L}=\mathcal{L}_{-\infty}$ is a left loop situated in a horizontal strip starting at the point $-\infty + \mathrm{i} \phi_1$ and terminating at the point $-\infty + \mathrm{i} \phi_2$ with $-\infty<\phi_1<\phi_2<+\infty$;
	\item $\mathcal{L}=\mathcal{L}_{+\infty}$ is a right loop situated in a horizontal strip starting at the point $+\infty + \mathrm{i} \phi_1$ and terminating at the point $+\infty + \mathrm{i}\phi_2$ with $-\infty<\phi_1<\phi_2<+\infty$;
	\item $\mathcal{L}=\mathcal{L}_{\mathrm{i}\gamma\infty}$ is a contour starting at the point $\gamma -\mathrm{i}\infty$ and terminating at the point $\gamma +\mathrm{i}\infty$, where $\gamma \in \mathbb{R}=(-\infty,+\infty)$.
\end{itemize}
In what follows, we assume that the contour of the Fox-$H$ functions used is of the kind $\mathcal{L}_{\mathrm{i}\gamma\infty}, \gamma \in (0,1)$.
\end{remark}
For further details on the existence of these functions; see \cite{SaiKil}.
As announced, gWfs could be represented by Fox-$H$ functions, depending on the values of the parameters $a_i,\alpha_i,b_j,\beta_j$, $i=1,\dots,m$, $j=1,\dots,p$. 
\begin{definition}[Generalized Wright function]\label{def:genWrigfuncti}
For $m,p \in \N_0$, $A_j,B_i\in \mathbb{C}$ and $\alpha_j,\beta_i \in \R \backslash \{0\}$, with $j=1,\dots,m$ and $i=1,\dots,p$, the gWfs are formally defined by the power series:
\[ 
_m\Psi_p\left[ \genfrac{}{}{0pt}{}{(B_i, \beta_i )_{1,m}}{(A_j, \alpha_j )_{1,p} }\Bigg| z \right]=\sum_{k\geq0} \frac{\prod_{i=1}^m \Gamma(B_i+\beta_i k)}{\prod_{j=1}^p \Gamma(A_j+\alpha_j k)} \frac{z^k}{k!}.
\]
They can be represented by Fox-$H$ functions if $\alpha_i,\beta_j$ are positive: 
\begin{equation}\label{eq:defGenWrightfuncasFoxHfunc}  
_m\Psi_p\left[ \genfrac{}{}{0pt}{}{(B_i, \beta_i )_{1,m}}{(A_j, \alpha_j )_{1,p} }\Bigg| -z \right]=H^{1,m}_{m,p+1}\left[z \, \Bigg| \genfrac{}{}{0pt}{}{(1-B_i, \beta_i )_{1,m}}{(0,1),(1-A_j, \alpha_j )_{1,p} }\right], \quad z \in \mathbb{C} \end{equation}
 see also Equation (5.2) in \cite{KST02}.
\end{definition}
\begin{remark} Along the paper, we recover the case of white noise analysis and the case of Mittag-Leffler analysis by imposing specific values to the parameters as follows:
\begin{itemize}
    \item The Mittag-Leffler function, $E_{\rho}(\cdot)$, $\rho \in (0,1]$, can be represented by gWfs by choosing  $m=1$, $p=1$, $B_1=1$, $\beta_1=1$, $A_1=1$, $\alpha_1=\rho$; see Section~6 in \cite{KST02} and Example 5.5 in \cite{BCDS23};
    \item The white noise case is obtained by imposing $m=p=0$; see Corollary 4.6 in \cite{BCDS23}.
    \end{itemize}
    In the following sections it is understood that the white noise case is a special case of Mittag Leffler analysis, hence of generalized Wright analysis, since for $\beta=1$ we have $E_{1}(\cdot)=\exp(\cdot)$.
\end{remark}

In what follows, we give assumptions on the parameters $m$, $p$, $a_i$, $b_j$, $\alpha_i$, $\beta_j$ to identify the class of Fox-$H$
densities with finite moments of any order. This class is used to find gWfs.
\begin{assumption}\label{ass:AllHdensity}
    Let $m,p \in \mathbb{N}_0:=\{0,1,2,\dots\}$ be given. We make the following assumptions on the parameters of the Fox-$H$ function:
\begin{enumerate}
\item Let $a_i \in \mathbb{R}$ and $\alpha_i>0$, for $i=1,\dots,p$, and $a_i+\alpha_i>0$, for $i=1,\dots,p$;
\item Let $b_j \in \mathbb{R}$ and $\beta_j>0$, for $j=1,\dots,m$,  and $b_j + \beta_j>0$, for $j=1,\dots,m$;
\item Let either $a^*>0$ or $a^*=0$ and $\mu<-1$;
\item Let $H^{m,0}_{p,m} \left[ \cdot \,\big|\genfrac{}{}{0pt}{}{(a_i,\alpha_i)_{1,p}}{(b_j,\beta_j)_{1,m}} \right]$ be non-negative on $(0,\infty)$.
\end{enumerate}

   where $a^*:=\sum_{j=1}^m \beta_j - \sum_{i=1}^p \alpha_i$ and $\mu:=\sum_{j=1}^m b_j -\sum_{i=1}^p a_i - (p-m)/2$.
\end{assumption}

\begin{lemma}[Fox-$H$ densities with all moments; see Lemma 3.4 in \cite{BCDS23}]\label{lem:HFiniteDensities}
			Let the Assumption~\ref{ass:AllHdensity}. Then the corresponding Fox-$H$ density	
\begin{equation}\label{eq:densitiesFiniteMoments}  
\varrho(\tau)=\frac{1}{K}H^{m,0}_{p,m} \left[ \tau \,\Bigg|\genfrac{}{}{0pt}{}{(a_i,\alpha_i)_{1,p}}{(b_j,\beta_j)_{1,m}} \right], \quad \tau>0,\end{equation}
where 
\begin{equation}\label{eq:norm-constant}
K:=\frac{\prod_{i=1}^m \Gamma(b_i+\beta_i)}{\prod_{i=1}^p \Gamma(a_i+\alpha_i)}     \end{equation} 
has finite moments of all orders. The moments are given by 
\[
\int_{0}^\infty \tau^l  \varrho(\tau)\, \mathrm{d}\tau
=\frac{1}{K} \frac{\prod_{i=1}^m \Gamma(b_i+\beta_i(l+1))}{\prod_{i=1}^p \Gamma(a_i+\alpha_i(l+1))}, \quad l=0,1,\dots.
\]
Furthermore, its LT is given by 
\begin{equation}\label{eq:LTdensitiesFiniteMoments} 
(\mathscr{L}\varrho)(s)
=\frac{1}{K}\, _m\!\Psi_p\left[ \genfrac{}{}{0pt}{}{(b_i+\beta_i,\beta_i)_{1,m}}{(a_j+\alpha_j,\alpha_j)_{1,p}} \middle| -s\right],   \end{equation}
where $s\geq 0$. Furthermore, when $a^* \in (0,1)$, $(\mathscr{L}\varrho)(\cdot)$ can be extended to an entire function.\\  
The above density $\varrho(\cdot)$ is called Fox-$H$ density with all moments (F$H$dam, hereafter) and we denote by $\mathcal{X}$ the class of random variables that have an F$H$dam as density.
\end{lemma}
We denote by $\mathcal{X}$ the class of random variables with Fox-$H$ density $\varrho(\cdot)$ given in Equation~\eqref{eq:densitiesFiniteMoments}.    

\begin{remark}
    Given that our characteristic functional $_m \Psi_q(\cdot)$ is defined by the mixture of Gaussian measures with F$H$dam, the question of infinite divisibility of this family is based on the infinite divisibility of F$H$dam, as stated in Proposition~2.2 in \cite{SVH03} and, in particular, Equation~(2.7) following it. Moreover, it is widely recognized that there exists a strong correlation between complete monotonicity and infinite divisibility, as evidenced by Theorem~51.6 in \cite{Sat99}. Complete monotonicity has proven to be an essential property in establishing the infinite divisibility of densities represented by special functions, as demonstrated in \cite{BPSV24}. In Appendix~\ref{app:Completely monotone FHDAM}, additional considerations on this topic are presented, grounded in the assumptions outlined in Lemma~\ref{lem:HFiniteDensities}.
\end{remark}

\begin{remark}\begin{enumerate}
    \item  Additional examples of Fox-$H$ densities of F$H$dam can be found in \cite{BCDS23};
    \item We will employ the following alternative short notation for the gWf we use when necessary:
\[ 
_m\!\Psi_p(\cdot):= \, _m\!\Psi_p\left[ \genfrac{}{}{0pt}{}{(b_i+\beta_i,\beta_i)_{1,m}}{(a_j+\alpha_j,\alpha_j)_{1,p}} \middle| \cdot \right] .
\]
\end{enumerate}
\end{remark}

To introduce gWm on $\R^d$, we recall that the gWfs $_m\!\Psi_p(\cdot)$  defined in the above lemma are completely monotone. {\color{blue} Furthermore, by assuming $a^*\in(0,1)$ the gWfs are entire, allowing their use as characteristic functions. Indeed}, by Corollary~\ref{cor:WrightCharactFunct} below, the following map is a characteristic function of $x \in \R^d:$
\begin{equation}\label{eq:CF-gWm-finite-d} \R^d \ni x \mapsto \frac{1}{K}\, _m\!\Psi_p\left[ \genfrac{}{}{0pt}{}{(b_i+\beta_i,\beta_i)_{1,m}}{(a_j+\alpha_j,\alpha_j)_{1,p}} \middle| -\frac{(x,x ) }{2} \right], \end{equation} where $(\cdot,\cdot)$ denotes the inner-product in $\R^d$. 
  
\begin{corollary}\label{cor:WrightCharactFunct}
		Let the assumptions of Lemma~\ref{lem:HFiniteDensities} hold. Then 
\[
\R^d \ni y \mapsto \frac{1}{K}\, _m\!\Psi_p\left[ \genfrac{}{}{0pt}{}{(b_i+\beta_i,\beta_i)_{1,m}}{(a_j+\alpha_j,\alpha_j)_{1,p}} \middle| -\frac{(y,y)}{2}\right] 
\]
is a characteristic function on $\R^d$.
		\end{corollary}

  \begin{proof}
  The continuity and normalization of $\frac{1}{K} \, _m \Psi_p(\cdot)$ are clear from Lemma~3.3 and Lemma~3.4 in \cite{BCDS23}.
  The property of being positive-definite is proved as follows.  For $n \in \N$, $y_1,\dots,y_n \in \R^d$ and $z_1,\dots,z_n \in \C$ we have 
  \begin{eqnarray*}  
  &&\sum_{i,j=1}^n z_i C(y_i-y_j)\bar{z}_j \\
  &=&\frac{1}{K} \sum_{i,j=1}^n  z_i\left( \, _m\!\Psi_p\left[ \genfrac{}{}{0pt}{}{(b_i+\beta_i,\beta_i)_{1,m}}{(a_j+\alpha_j,\alpha_j)_{1,p}} \middle| -\frac{(y_i-y_j,y_i-y_j)}{2}\right]\right)\bar{z}_j\\
  &=&\frac{1}{K} \int_0^\infty \sum_{i,j=1}^n  z_i \e^{-\tau (y_i-y_j,y_i-y_j)/2} \bar{z}_j H^{m,0}_{p,m} \left[ \tau \,\Bigg|\genfrac{}{}{0pt}{}{(a_i,\alpha_i)_{1,p}}{(b_j,\beta_j)_{1,m}} \right]\mathrm{d}\tau\\ 
  &=& \frac{1}{K} \int_0^\infty \sum_{i,j=1}^n  z_i C_{\mu_{\tau I}}(y_i-y_j) \bar{z}_j H^{m,0}_{p,m} \left[ \tau \,\Bigg|\genfrac{}{}{0pt}{}{(a_i,\alpha_i)_{1,p}}{(b_j,\beta_j)_{1,m}} \right]\mathrm{d}\tau\geq 0,\end{eqnarray*}
    where $C_{\mu_{\tau I}}(\cdot)$ is the characteristic function of a Gaussian random variable in $\R^d$ with zero mean and covariance matrix $\tau I$, where $I$ is the identity matrix.
  \end{proof}

\begin{remark}\label{rem:Wright'sCharacteFunctionEntire}
Under the assumptions of Lemma~\ref{lem:HFiniteDensities}  we obtain the series representation of $_m\!\Psi_p$ with $z \in \C$:
\begin{equation}\label{eq:PsiSeries} _m\!\Psi_p\left[ \genfrac{}{}{0pt}{}{(b_i+\beta_i,\beta_i)_{1,m}}{(a_j+\alpha_j,\alpha_j)_{1,p}} \middle| -\frac{ z }{2} \right]=\sum_{n \geq 0} \frac{\prod_{i=1}^m \Gamma(b_i + \beta_i (n+1))}{\prod_{i=1}^p \Gamma(a_i + \alpha_i (n+1))} \frac{(-z)^n}{n!2^n}, \end{equation}
see also \cite{KST02}. 
\end{remark}
 \subsection{Generalized Wright Measures on the Euclidean Spaces}
 Generalized Wright measures are defined via the Bochner theorem by using $_m\!\Psi_p(-(\cdot,\cdot)/2)$ as a characteristic function on $\R^d$ as follows,  as consequence of Corollary~\ref{cor:WrightCharactFunct}. Here we also give the formula for the mixed moments and compute its Radon-Nikodym derivative w.r.t.~the Lebesgue measure $\mathrm{d}x$, i.e.~$\mathrm{d}\mu^d_\Psi(x)=\varrho_H^d(x)\,\mathrm{d}x$, $x\in\R^d$.

\begin{definition}\label{def:finite-dim-gWm}
    Let $d \in \N$ be given and assume that Lemma~\ref{lem:HFiniteDensities} is valid. We define $\mu^d_\Psi$ as the unique probability measure in $\R^d$ with the characteristic function given by
    \[  \int_{\R^d} \exp(\mathrm{i} (y,x))\,\mathrm{d}\mu_{\Psi}^d(x)=\frac{1}{K}\, _m\!\Psi_p\left[ \genfrac{}{}{0pt}{}{(b_i+\beta_i,\beta_i)_{1,m}}{(a_j+\alpha_j,\alpha_j)_{1,p}} \middle| -\frac{(y,y ) }{2} \right],\quad  y \in \R^d,\]
    where $K$ is given in \eqref{eq:norm-constant}.
\end{definition}

We compute the mixed moments and explicitly show the Radon-Nikodym derivative of $\mu_\Psi^d$, denoted by $\varrho^d_H(\cdot)$.
It follows from Equation~\eqref{eq:PsiSeries} that the family of measures, $\mu_{\Psi}^d$, have moments of any order.

\begin{lemma}[Moments]
    The measure $\mu_{\Psi}^d$ has moments of any order. More precisely, if
\[
M^{\mu_{\Psi}^d}(n_1,\dots,n_d):= \int_{\R^d} x_1^{n_1} \dots x_d^{n_d} \,\mathrm{d}\mu_{\Psi}^d(x),  
\]
where $n_1,\dots,n_d \in \N_0$, then $M^{\mu_{\Psi}^d}(n_1,\dots,n_d)=0$ if any of $n_i$ is odd, and for even moments we have
\[  
M^{\mu_{\Psi}^d}(2n_1,\dots,2n_d)=\frac{1}{K}\frac{\prod_{i=1}^m \Gamma(b_i + \beta_i (n+1))}{\prod_{i=1}^p \Gamma(a_i + \alpha_i (n+1))} \frac{(2n_1)! \cdots (2n_d)! }{n_1! \dots n_d!2^{n}},
\]
where $n=n_1 + \dots + n_d$.
\end{lemma}

\begin{proof}
Let $\kappa=(\kappa_1, \dots, \kappa_d)\in \N_0^d$ and $f \in C^{\infty}(B_R(0))$ with $B_R(0) \subset \R^d$, $R>0$, be given. We introduce the multi-index notation to define the operator $\partial_{x}^{\kappa}$ with index $\kappa$ and $x=(x_1,\dots,x_d) \in \R^d$ on $C^{\infty}(B_R(0))$ as: 
\[
\partial_{x}^{\kappa} f:=\partial^{\kappa_1}_1 \dots \partial^{\kappa_d}_d f(x), 
\]
where $\partial^{\kappa_i}_i:=\partial^{\kappa_i} \, / \partial x^{\kappa_i}_i$, $i=1,\dots,d$.
 We compute the mixed moments of order $(\kappa_1,\dots, \kappa_d)$ by applying the operator $\mathrm{i}^{-|\kappa|}\partial_{y}^{\kappa}$ with index $\kappa$, $|\kappa|=\kappa_1+\dots+\kappa_d$ and $y=(y_1,\dots,y_n) \in \R^d$  on 
\[ 
\int_{\R^d} \exp\left(\mathrm{i} \left(\sum_{i=1}^d y_ie_i,x \right)\right)\mathrm{d}\mu_{\Psi}^d(x)=\frac{1}{K}\, _m\!\Psi_p\left[ \genfrac{}{}{0pt}{}{(b_i+\beta_i,\beta_i)_{1,m}}{(a_j+\alpha_j,\alpha_j)_{1,p}} \middle| -\frac{\left(y, y \right) }{2} \right]
\]
where $\{e_i\}_{i=1}^d$ is the canonical orthogonal basis of $\R^d$. The operator applied to the left-hand side of the above equation (at $y=0$) gives the mixed moments of order $(\kappa_1,\dots, \kappa_d)$, while the right-hand side gives
\begin{eqnarray*}
 \int_{\R^d} x_1^{\kappa_1} \dots x_d^{\kappa_d}\,\mathrm{d}\mu_{\Psi}^d(x)&=&\mathrm{i}^{-|\kappa|}\partial_{y}^{\kappa} \left( \frac{1}{K}\, _m\!\Psi_p\left[ \genfrac{}{}{0pt}{}{(b_i+\beta_i,\beta_i)_{1,m}}{(a_j+\alpha_j,\alpha_j)_{1,p}} \middle| -\frac{\left(y, y \right)}{2} \right] \right)\Bigg|_{y=0}\\
 &\overset{*}{=}&\frac{\mathrm{i}^{-|\kappa|}}{K}\sum_{l \geq 0} \frac{\prod_{i=1}^m \Gamma(b_i + \beta_i (l+1))}{\prod_{i=1}^p \Gamma(a_i + \alpha_i (l+1))} \frac{(-1)^l}{l!2^l}\partial_{y}^{\kappa} \left(\sum_{i=1}^d y_i ^2\right)^l \Bigg|_{y=0} \\
 &\overset{**}{=}&\frac{1}{K}\frac{\prod_{i=1}^m \Gamma(b_i + \beta_i (n+1))}{\prod_{i=1}^p \Gamma(a_i + \alpha_i (n+1))} \frac{ \prod_{i=1}^d(2n_i)!  }{n! 2^{n}} \dbinom{n}{n_1,\dots, n_d},
\end{eqnarray*}
where in $*$ we use the series representation of the function $_m\!\Psi_p(\cdot)$ given in Equation~\eqref{eq:PsiSeries} and in $**$ we used the equation below:
\[   \partial_{y}^{\kappa}\left(\sum_{i=1}^d y_i^2\right)^l\Bigg|_{y=0}=\begin{cases} 0, & \text{ if any of $\kappa_i$ is odd or }l\neq 2n,\\
            \dbinom{n}{n_1!\dots n_d!} \prod_{i=1}^d(2n_i)!, & \text{ otherwise} l=2n \end{cases}\]
with $n:=n_1+\dots+n_d$ where $\kappa_i=2n_i,\, i=1,\dots,d$, and
\[ 
\binom{n}{n_1,  \dots, n_d}:=\frac{n!}{n_1! \dots n_d!}.
\]
This concludes the proof.
\end{proof}
\begin{remark}
We recover the moments of the finite-dimensional Mittag-Leffler measure in Lemma~2.4 in \cite{JahnI}, imposing $m=1$, $p=1$, $b_1=0,\, \beta_1=1,\, a_1=1-\rho,\, \alpha_1=\rho$; see Section~6 in \cite{KST02} and Example 5.5 in \cite{BCDS23}.
\end{remark}
The next theorem establishes that each measure $\mu_\Psi^d$ is absolutely continuous w.r.t.~the Lebesgue measure on $\mathbb{R}^d$ and identifies its Radon-Nikodym derivative.
\begin{theorem}\label{thm:FiniteDimensionalFoxHGaussianDensity}
    For $d \in \N$, let Lemma~\ref{lem:HFiniteDensities} holds  together with $2(b_j + \beta_j)> \beta_jd$, for $j=1,\dots,m$.
    Then the $d$-dimensional generalized Wright measure $\mu_{\Psi}^d$ is absolutely continuous w.r.t.~the $d$-dimensional Lebesgue measure and its Radon-Nikodym density is expressed for every $x \in \R^d$ by
 \begin{equation}\label{eq:FiniteDimensionalFoxHGaussianDensity}  \varrho_H^d(x):= \frac{1}{(2 \pi)^{d/2}K} H^{m+1,0}_{p,m+1}\left[ \frac{(x,x)}{2}\,\bigg| \genfrac{}{}{0pt}{}{(a_i+\alpha_i(1-d/2),\alpha_i)_{1,p}}{(0,1),(b_{j-1}+\beta_{j-1}(1-d/2),\beta_{j-1})_{2,m+1} }\right]. 
    \end{equation}
\end{theorem}
\begin{proof}
Let $y \in \R^d$ be given. We have 
\begin{eqnarray*}
    &&\int_{\R^d} \exp(\mathrm{i} ( y, x))\,\mathrm{d}\mu_{\Psi}^d(x)\\
    &=&\frac{1}{K}\, _m\!\Psi_p\left[ \genfrac{}{}{0pt}{}{(b_i+\beta_i,\beta_i)_{1,m}}{(a_j+\alpha_j,\alpha_j)_{1,p}} \middle| -\frac{(y,y) }{2} \right]\\
    &=&\frac{1}{K}\int_0^\infty \exp\left(-\frac{(y,y) }{2} \tau\right) H^{m,0}_{p,m}(\tau)\, \mathrm{d}\tau\\
    &=&\frac{1}{K}\int_0^\infty \int_{\R^d}\exp(\mathrm{i}(y,x))  \frac{1}{(2 \pi \tau)^{d/2}}\exp\left(-\frac{(x,x)}{2 \tau} \right) \mathrm{d}x H^{m,0}_{p,m}(\tau)\, \mathrm{d}\tau\\
    &\overset{*}{=}&\frac{1}{K}\int_{\R^d} \exp(\mathrm{i}(y,x)) \int_0^\infty  \frac{1}{(2 \pi \tau)^{d/2}} H^{0,1}_{1,0}\left[ \frac{2 \tau}{(x,x)}\,\bigg| \genfrac{}{}{0pt}{}{(1,1)}{-\!-}\right]  H^{m,0}_{p,m}(\tau)\,\mathrm{d}\tau\,  \mathrm{d}x\\
    &\overset{**}{=}&\int_{\R^d} \exp(\mathrm{i}(y,x))\varrho_H^d(x) \mathrm{d}x
    \end{eqnarray*}
where in $*$ we used Equations~(2.9.4) and (2.1.3) in \cite{SaiKil} and in $**$ we use $2(b_j + \beta_j)> \beta_jd$, $j=1,\dots,m$ to employ Equation~(2.8.4) in \cite{SaiKil} if $a^*>0$, otherwise Equation (2.8.12) in \cite{SaiKil} if $a^*=0$ holds with $\mu<-1$. Finally, we exploit Equations~(2.1.3) and (2.1.5) in \cite{SaiKil} to obtain \eqref{eq:FiniteDimensionalFoxHGaussianDensity}.
\end{proof}

\begin{remark}
\begin{enumerate}
\item Equation~\eqref{eq:FiniteDimensionalFoxHGaussianDensity} coincides with Equation~(3.23) in \cite{S92} for $b_1=0,\, \beta_1=1,\, a_1=1-\rho,\, \alpha_1=\rho$ which is the density of the $d$-dimensional Mittag-Leffler measure;
\item  The additional assumptions on the parameters $b_j,\beta_j,j=1,\dots,m$ in Theorem~\ref{thm:FiniteDimensionalFoxHGaussianDensity} are consistent with the hypothesis in Assumption~\ref{ass:AllHdensity}.
\item In the above theorem, Equation (2.8.12) is needed as an extension of Equation (2.8.4) in \cite{SaiKil} if $a^*=0$ holds with $\mu<-1$, instead of $a^*>0$.
\end{enumerate}
\end{remark}

Measures in class $\mathcal{M}_\Psi(\mathbb{R}^d)$ can be expressed as a mixture of Gaussian measures with a proper probability measure on $(0,\infty)$. More precisely, we have

\begin{corollary}\label{cor:mixture-fd}
Under the hypothesis of Theorem \ref{thm:FiniteDimensionalFoxHGaussianDensity} the density $\varrho_H^d(\cdot)$ is a Gaussian mixture as follows:
\begin{align*}
      \varrho_H^d(x)&=\frac{1}{(2 \pi)^{d/2}K} H^{m+1,0}_{p,m+1}(x)\\
      &=\frac{1}{(2 \pi)^{d/2}K}\int_0^\infty  \frac{1}{\tau^{d/2}} \exp\left(-\frac{(x,x)}{2 \tau} \right)  H^{m,0}_{p,m}(\tau)\,\mathrm{d}\tau, \quad x \in \R^d  .
\end{align*}
\end{corollary}

\begin{remark}[Multivariate Elliptical Distribution]
 It follows from Theorem~\ref{thm:FiniteDimensionalFoxHGaussianDensity} that the measures $\mu^d_\Psi$, $d\in\N$, are multivariate elliptical distributions, written as $E_d(0,I,\, _m\Psi_p(-\cdot))$, whose density generator is
\[ 
 H^{m+1,0}_{p,m+1}(x)=\int_0^\infty  \frac{1}{\tau^{d/2}} \exp\left(-\frac{(x,x)}{2 \tau} \right) H^{m,0}_{p,m}(\tau)\,\mathrm{d}\tau, \quad x \in \R^d, 
\]
see \cite{DaSil15}, \cite{Fan90} and \cite{Steu04}, Chap.~VI for a deeper insight into multivariate elliptical distributions. 
\end{remark}

\subsection{Fox-Hermite Polynomials}
We apply Gram-Schmidt orthogonalization to monomials $x^n,\, n \in \N_0$, to obtain monic polynomials $H^{\mu_{\Psi}^1}_n$, $n \in \N_0$, with $\mathrm{deg}\,H^{\mu_{\Psi}^1}_n=n$, which are orthogonal in $L^2(\R,\mu_{\Psi}^1)$. These polynomials are determined by the moments of the measure $\mu_{\Psi}^1$. The first four of these polynomials are given by 

\begin{eqnarray*} H^{\mu_{\Psi}^1}_0(x)&=&1,\\ 
H^{\mu_{\Psi}^1}_1(x)&=&x,\\ 
H^{\mu_{\Psi}^1}_2(x)&=&x^2-\frac{1}{K}\frac{\prod_{i=1}^m \Gamma(b_i+2\beta_i)}{\prod_{i=1}^p \Gamma(a_i+2\alpha_i)},\\ 
H^{\mu_{\Psi}^1}_3(x)&=&x^3-3\frac{\prod_{i=1}^m \Gamma(b_i+3\beta_i)}{\prod_{i=1}^p \Gamma(a_i+3\alpha_i)}\frac{\prod_{i=1}^p \Gamma(a_i+2\alpha_i)}{\prod_{i=1}^m \Gamma(b_i+2\beta_i)} x. \end{eqnarray*}

\begin{remark}
    Note that for $m=1$, $p=1$, $b_1=0$, $\beta_1=1$, $a_1=1-\rho$ and $\alpha_1=\rho$, the measure $\mu_{\Psi}^d$ is the finite-dimensional Mittag-Leffler measure, $\mu_\rho^1$, hence, in this case, these polynomials coincide with their homologous in space $L^2(\R,\mu_\rho^1)$; see \cite{JahnI}.
\end{remark}


Now we define a new class of polynomials and their generating function, which will be used to decompose the Laplace transform of a measure $\mu_\Psi \in \mathcal{M}_{\mu_\Psi}(\mathcal{N})$ and the $T_{\mu_\Psi}$-transform of Donsker's delta in Section~\ref{sec:DonskDelta}. For a review of Chebyschev-Hermite polynomials (also called Hermite polynomials); see \cite{Bog22}. We recall the definition of the Hermite polynomials in $x \in \R$ of degree $n$ and parameter $\sqrt{\tau}$, $\tau>0$,
\[  
H^{\tau}_n(x):=(-\tau)^n \exp\left(\dfrac{x^2}{2\tau}\right)  \dfrac{\mathrm{d}^n}{\mathrm{d} x^n} \exp\left( -\dfrac{x^2}{2\tau} \right), \quad n=0,1,2,\dots,  
\]
see Equation (A.1.13) in \cite{HKPS93}.

\begin{lemma}\label{lem:FoxHermitePolyPropert}
    Let the assumptions of Lemma \ref{lem:HFiniteDensities} hold. Then the following functions are monic polynomials of degree $n$ on $\R$
    \begin{equation}\label{eq:FoxHermitePolynomials}  F^H_{n}(x):=\frac{1}{K}\int_0^\infty  H^{\tau}_n \left(x\right)H^{m,0}_{p,m}\left[ \tau\,\bigg| \genfrac{}{}{0pt}{}{(a_i,\alpha_i)_p}{(b_j,\beta_j)_m}\right]  \mathrm{d}\tau, \quad n=0,1,\dots, \; x \in \R, \end{equation}
    where $H^{\tau}_n(x)$ are Hermite polynomials of degree $n$ and parameter $\sqrt{\tau}$. They have the following properties:

    \begin{enumerate}

    \item The following representation holds
\[ 
F^H_{n}(x)=\sum_{k=0}^{[n/2]} \frac{(-1)^k}{K}\frac{\prod_{i=1}^m \Gamma(b_i + \beta_i (k+1))}{\prod_{i=1}^p \Gamma(a_i + \alpha_i (k+1))}\binom{n}{2k}  (2k-1)!!x^{n-2k},
\]
    where $[n/2]$ is equal to $n/2$ if $n$ is even, otherwise it is equal to $(n-1)/2$;
    \item They satisfy $D_x F^H_{n}(x)=n F^H_{n-1}(x)$;
    \item They are generated by
    \[  e_H(x;t):= \frac{\exp(tx)}{K}\, _m\!\Psi_p\left[ \genfrac{}{}{0pt}{}{(b_i+\beta_i,\beta_i)_{1,m}}{(a_j+\alpha_j,\alpha_j)_{1,p}} \middle| -\frac{ t^2}{2} \right].\]
    \end{enumerate}
    We call $F_n^H(\cdot)$ Fox-Hermite polynomials. 
\end{lemma}

\begin{proof}

\begin{enumerate}
\item For $n \in \N_0$, let $H^{\tau}_n(\cdot)$ be the Hermite polynomials of degree $n$ and parameter $\sqrt{\tau}$, then 

\begin{eqnarray*}
    F_n^H(x)&=&\frac{1}{K}\int_0^\infty  H^{\tau}_n \left(x\right)H^{m,0}_{p,m}\left[ \tau\,\bigg| \genfrac{}{}{0pt}{}{(a_i,\alpha_i)_p}{(b_j,\beta_j)_m}\right] \mathrm{d}\tau\\
    &\overset{*}{=}&\sum_{k=0}^{[n/2]} (-1)^k\binom{n}{2k}  \frac{(2k-1)!!x^{n-2k}}{K}\int_0^\infty \tau^k  H^{m,0}_{p,m}\left[ \tau\,\bigg| \genfrac{}{}{0pt}{}{(a_i,\alpha_i)_p}{(b_j,\beta_j)_m}\right] \mathrm{d}\tau\\
    &=&\sum_{k=0}^{[n/2]} (-1)^k\binom{n}{2k}  \frac{(2k-1)!!x^{n-2k}}{K}\int_0^\infty \tau^k  H^{m,0}_{p,m}\left[ \tau\,\bigg| \genfrac{}{}{0pt}{}{(a_i,\alpha_i)_p}{(b_j,\beta_j)_m}\right] \mathrm{d}\tau\\
    &\overset{**}{=}& \sum_{k=0}^{[n/2]} \frac{(-1)^k}{K}\frac{\prod_{i=1}^m \Gamma(b_i + \beta_i (k+1))}{\prod_{i=1}^p \Gamma(a_i + \alpha_i (k+1))}\binom{n}{2k}  (2k-1)!!x^{n-2k}  \\
\end{eqnarray*}
where in $*$ we used Equations~(A.1.16) in \cite{HKPS93} and in $**$ we used Lemma~\ref{lem:HFiniteDensities}. We note that in Equations (A.1.16) in \cite{HKPS93} the term $[n/2]$ has the meaning of $n/2$ or $(n-1)/2$ depending on whether $n$ is even or odd, respectively. It is straightforward to deduce that $F_n^H(x)$ are monic polynomials of $n$ for $x \in \R$, specifically
\[  F^H_n(x)=x^n + \sum_{k=1}^{[n/2]} \frac{(-1)^k}{K}\frac{\prod_{i=1}^m \Gamma(b_i + \beta_i (k+1))}{\prod_{i=1}^p \Gamma(a_i + \alpha_i (k+1))}\binom{n}{2k}  (2k-1)!!x^{n-2k}, \]
taking into account $K$ given in \eqref{eq:norm-constant}.
\item The property $D_x F^H_{n}(x)=n F^H_{n-1}(x)$ follows directly by applying Equation~(A.1.20) in \cite{HKPS93} to Equation~\eqref{eq:FoxHermitePolynomials} interchanging the derivative and integral.
\item The function $e_H(x;t)$ can be expressed by 
\begin{eqnarray*}
e_H(x;t)&\overset{*}{=}&\exp(tx) \frac{1}{K}\int_0^\infty \exp(-\tau t^2/2)H^{m,0}_{p,m}\left[ \tau\,\bigg| \genfrac{}{}{0pt}{}{(a_i,\alpha_i)_p}{(b_j,\beta_j)_m}\right]\mathrm{d}\tau\\
&=&\frac{1}{K}\int_0^\infty \exp(tx-\tau t^2/2)H^{m,0}_{p,m}\left[ \tau\,\bigg| \genfrac{}{}{0pt}{}{(a_i,\alpha_i)_p}{(b_j,\beta_j)_m}\right] \mathrm{d}\tau \\
&\overset{**}{=}& \frac{1}{K}\int_0^\infty \sum_{n \geq 0}\frac{t^n}{n!} H^{\tau}_n (x)H^{m,0}_{p,m}\left[ \tau\,\bigg| \genfrac{}{}{0pt}{}{(a_i,\alpha_i)_p}{(b_j,\beta_j)_m}\right] \mathrm{d}\tau\\
&=&\sum_{n \geq 0} \frac{t^n}{n!}F^H_n(x),
\end{eqnarray*}
where we used Lemma~\ref{lem:HFiniteDensities} and the definition of $_m\!\Psi_p(\cdot)$ in $*$ and Equation~(A.1.15) in \cite{HKPS93} in $**$.
\end{enumerate}
\end{proof}

\section{Generalized Wright Measures}
\label{Sec:Fox-HAnalysis}
In this section, we define the class of probability measures $\mathcal{M}_\Psi(\mathcal{N}')$ on the co-nuclear space $\mathcal{N}'$ of a nuclear space $\mathcal{N}$.  Therefore, our starting point (see \cite{GV1968}) is an abstract nuclear Gelfand triple
\begin{equation}\label{eq:abs-nuclear-triple}
\mathcal{N}\subset \mathcal{H}\subset\mathcal{N}'.    
\end{equation}
We define the so-called generalized Wright probability space $(\mathcal{N}',C_\sigma(\mathcal{N}'), \mu_\Psi)$ by employing gWf $ _m \Psi_p(\cdot)$ in the characteristic functional. Hence, we define non-Gaussian Hilbert spaces $L^2(\mu_\Psi)$ and compute the mixed moments of $\mu_\Psi$. 

Let $\mathcal{H}$ be a real separable Hilbert space with inner product $(\cdot,\cdot)$ and corresponding norm $\left|  \cdot\right|$. Let us take $\mathcal{N}$ as a nuclear space that is densely and continuously embedded in $\mathcal{H}$. We denote by $\mathcal{N}'$  its dual space. Consequently, the canonical dual pairing between $\mathcal{N}'$ and $\mathcal{N}$ is denoted by $\langle\cdot,\cdot\rangle$, and if we identify $\mathcal{H}$ with its dual space via the Riesz isomorphism, we get the chain of inclusions $\mathcal{N}\subset\mathcal{H}\subset\mathcal{N}'$. The dual pairing is a bilinear extension of the inner product in $\mathcal{H}$, in fact, $\langle f,\varphi\rangle=(f,\varphi)$ for $f\in\mathcal{H}$, $\varphi\in\mathcal{N}$.

We assume also that $\mathcal{N}$ is represented by a countable family of Hilbert spaces: for each $l\in\N_0$ let $\mathcal{H}_l$ be a real separable Hilbert space with norm $\left|  \cdot\right|_l$ such that $\mathcal{N}\subset\mathcal{H}_{l+1}\subset\mathcal{H}_l\subset\mathcal{H}$ continuously and the inclusions $\mathcal{H}_{l+1}\subset\mathcal{H}_l, l \in \N_0$ are Hilbert--Schmidt operators. There is no loss of generality in assuming $\left|  \cdot\right|_l\leq \left|  \cdot\right|_{l+1}$ on $\mathcal{H}_{l+1}$ and $\mathcal{H}_0=\mathcal{H}$, $\left|  \cdot\right|_0=\left|  \cdot\right|$. Then, the space $\mathcal{N}$ turns out to be the projective limit of the spaces $(\mathcal{H}_l)_{l\in\N}$, that is, $\mathcal{N}=\bigcap_{l\in\N_0}\mathcal{H}_l$ and the topology on $\mathcal{N}$ is the coarsest topology such that all inclusions $\mathcal{N}\subset\mathcal{H}_l$ are continuous.

This also gives a representation of $\mathcal{N}'$ in terms of an inductive limit. Let $\mathcal{H}_{-l}$ be the dual space of $\mathcal{H}_l$ with respect to $\mathcal{H}$ and let the dual pairing between $\mathcal{H}_{-l}$ and $\mathcal{H}_l$ be indicated by $\langle\cdot,\cdot\rangle$ as well. Then $\mathcal{H}_{-l}$ is a Hilbert space and we use $\left|  \cdot\right|_{-l}$ to express its norm. It follows from general duality theory that $\mathcal{N}'=\bigcup_{l\in\N_0}\mathcal{H}_{-l}$, and we equip $\mathcal{N}'$ with the inductive topology, that is, the finest locally convex topology such that all inclusions $\mathcal{H}_{-l}\subset\mathcal{N}'$ are continuous.

We end up with the following chain of dense and continuous inclusions:
\[ \mathcal{N}\subset\mathcal{H}_{l+1}\subset\mathcal{H}_l\subset\mathcal{H}\subset\mathcal{H}_{-l}\subset\mathcal{H}_{-(l+1)}\subset\mathcal{N}'. \]
For all the real spaces above we also consider their complexifications, which will be distinguished by a subscript $\C$, e.g. the complexification of $\mathcal{H}_l$ is $\mathcal{H}_{l,\C}$ and so on. In the following,    we always identify $f = [f_1,f_2]\in\mathcal{H}_{l,\C}, f_1,f_2\in\mathcal{H}_l$ for $l\in\Z$ with $f=f_1 + \mathrm{i}f_2$. The dual pairing extends as a bilinear form to $\mathcal{N}'_{\mathbb{C}}\times\mathcal{N}_{\mathbb{C}}$.


Consider the map $C : \mathcal{N} \longrightarrow \C$ defined by
\[   
C(\varphi):=\frac{1}{K}\, _m\!\Psi_p\left[ \genfrac{}{}{0pt}{}{(b_i+\beta_i,\beta_i)_{1,m}}{(a_j+\alpha_j,\alpha_j)_{1,p}} \middle| -\frac{\langle \varphi,\varphi \rangle }{2} \right], \quad \varphi \in \mathcal{N}, 
\] 
with normalization constant $K$ given in \eqref{eq:norm-constant}. The map $C(\cdot)$ is a characteristic functional on $\mathcal{N}$. In fact, from Corollary~\ref{cor:WrightCharactFunct} it is positive definite by substituting the inner product on $\R^d$ by the inner product on $\mathcal{H}$.  
Thus, using the Bochner-Minlos theorem, there exists a unique probability measure $\mu_\Psi$ on $\mathcal{N}'$, equipped with the cylindrical $\sigma$-algebra $\mathcal{C}_\sigma(\mathcal{N}')$, such that its Fourier transform is equal to $C(\cdot)$.  
This motivates the following definition.

\begin{definition}[Generalized Wright measures] \label{def:FoxHmeasu}
Under the conditions of Lemma~\ref{lem:HFiniteDensities}, the generalized Wright measure (gWm for short) $\mu_{\Psi}$ is defined as the unique probability measure on $\mathcal{N}'$ such that  
\begin{equation} \label{eq:FoxHmeasure}   \int_{\mathcal{N}'}\mathrm{e}^{\mathrm{i}\langle \omega,\varphi \rangle}\,\mathrm{d}\mu_{\Psi}(\omega)=\frac{1}{K}\, _m\!\Psi_p\left[ \genfrac{}{}{0pt}{}{(b_i+\beta_i,\beta_i)_{1,m}}{(a_j+\alpha_j,\alpha_j)_{1,p}} \middle| -\frac{\langle \varphi,\varphi \rangle }{2} \right], \quad \varphi \in \mathcal{N}, \end{equation} 
where $K$ is given in \eqref{eq:norm-constant}.
We denote this class of measures by $\mathcal{M}_\Psi(\mathcal{N}')$ and the generalized Wright probability space by $(\mathcal{N}',\mathcal{C}_\sigma(\mathcal{N}'),\mu_{\Psi})$. The corresponding $L^p$ Banach spaces of complex-valued $\mathcal{C}_\sigma(\mathcal{N}')$-measurable  functions with integrable $p$-th power are denoted by $L^p(\mu_{\Psi}):=L^p\big(\mathcal{N}',\mathcal{C}_\sigma(\mathcal{N}'),\mu_{\Psi};\C\big)$, $p\geq 1$. The norm in $L^p(\mu_{\Psi})$ is denoted by $\| \cdot\|_{L^p(\mu_{\Psi})}$.
	\end{definition} 

 \begin{remark}
The class $\mathcal{M}_\Psi(\mathcal{N}')$ of gWm covers many well-known measures in the literature depending on the realization of the triple $\mathcal{N}\subset\mathcal{H}\subset\mathcal{N}'$ and the choice of the parameters of $_m\!\Psi_p$. For example, Mittag-Leffler measures, Gaussian measures, and fractional Gaussian measures; see \cite{JahnI} and \cite{DU99}. In addition, the class $\mathcal{M}_\Psi(\mathcal{N}')$ is a subset of analytic measures $\mathcal{M}_a(\mathcal{N}')$ (cf.~\cite{KSWY98}); see Section~\ref{sec:AnalypropFHmeasu} below.
 \end{remark}

The first elementary property of the elements in $\mathcal{M}_\Psi(\mathcal{N}')$ is the relationship with the finite-dimensional generalized Wright measures $\mu_\Psi^d$; see Definition~\ref{def:finite-dim-gWm}.

\begin{lemma}\label{lem:ProjeMoments}
    Let $\varphi_1,\dots,\varphi_d \in \mathcal{N}$ be orthonormal in $\mathcal{H}$, then the image measure of $\mu_{\Psi}$ under the mapping $\mathcal{N}' \ni \omega \mapsto (\langle \omega , \varphi_1 \rangle,\dots, \langle \omega , \varphi_d \rangle)\in\R^d$ is the finite-dimensional gWm $\mu_{\Psi}^d$.
\end{lemma}

\begin{proof}
Let $\varphi_1,\dots,\varphi_d \in \mathcal{N}$ be orthonormal in $\mathcal{H}$ and denote by $T$ the measurable map $T:(\mathcal{N}',\mathcal{C}_{\sigma}(\mathcal{N}'),\mu_\Psi)\longrightarrow(\R^d, \mathcal{B}(\R^d))$ defined by
\[
T(\omega):=\big(\langle \omega , \varphi_1 \rangle,\dots, \langle \omega , \varphi_d \rangle \big).
\] 
We denote by $T_* \mu_\Psi:=\mu_\Psi \circ T^{-1}$ the image of $\mu_\Psi$ under the map $T$.
It is sufficient to show that both measures ($\mu_\Psi^d$ and $T_*\mu_\Psi$) have the same Fourier transform. For $y=(y_1,\dots,y_d)\in \R^d$, we have

\begin{equation}\label{eq:image-measure}
    \int_{\R^d}\mathrm{e}^{\mathrm{i}(x,y)_{\R^d}}\,\mathrm{d}(T_* \mu_\Psi)(x)=\int_{\mathcal{N}'}\mathrm{e}^{\mathrm{i}( T \omega,y)_{\R^d}}\,\mathrm{d}\mu_{\Psi}(\omega) = \int_{\mathcal{N}'}\mathrm{e}^{\mathrm{i}\langle \omega, \eta\rangle}\,\mathrm{d}\mu_{\Psi}(\omega),
\end{equation}
where $\eta:=\sum_{i=1}^d y_i\varphi_i$. The right-hand side integral is given by the characteristic functional of $\mu_\Psi$ as
\[
\frac{1}{K}\, _m\!\Psi_p\left[ \genfrac{}{}{0pt}{}{(b_i+\beta_i,\beta_i)_{1,m}}{(a_j+\alpha_j,\alpha_j)_{1,p}} \middle| -\frac{\langle \eta,\eta\rangle }{2} \right].
\]
Using the orthonormal property of $\varphi$'s, we obtain $\langle \eta,\eta\rangle=|y|^2$. Thus, we showed 
\[
\int_{\R^d}\mathrm{e}^{\mathrm{i}(x,y)}\,\mathrm{d}(T_* \mu_\Psi)(x) = \frac{1}{K}\, _m\!\Psi_p\left[ \genfrac{}{}{0pt}{}{(b_i+\beta_i,\beta_i)_{1,m}}{(a_j+\alpha_j,\alpha_j)_{1,p}} \middle| -\frac{ (y,y) }{2} \right].
\]
The right-hand side of this equality coincides with the characteristic function of $\mu_\Psi^d$; see Equation~\eqref{eq:CF-gWm-finite-d}.
\end{proof}

Next, we would like to compute the moments of the measures $\mu_\Psi\in\mathcal{M}_\Psi(\mathcal{N}')$ explicitly. To this end, we denote by $M_{n}^{\mu_{\Psi}}(\cdot)$, $n\in\mathbb{N}_0$, the moment or order $n$ of $\mu_{\Psi}$.

\begin{definition}[Mixed moments]\label{def:moments}
Let $n\in\mathbb{N}$, $\mu_{\Psi}\in\mathcal{M}_\Psi(\mathcal{N}')$, and $\varphi_i\in \mathcal{N}$, $i=1,\dots,n$, be given. The generalized moments of $\mu_{\Psi}$ are defined by
\[
 M_n^{\mu_{\Psi}}(\varphi_1,\dots, \varphi_n):=\int_{\mathcal{N}'}\langle \omega^{\otimes n}, \varphi_1\otimes\dots\otimes \varphi_n \rangle\, \mathrm{d}\mu_{\Psi}(\omega).
\]
\end{definition}
The result of Lemma~\ref{lem:ProjeMoments} implies directly the following. 

\begin{theorem}\label{thm:MomentmuHInfiniteDimension}
 Let the assumptions of Lemma \ref{lem:HFiniteDensities} hold, $\varphi \in \mathcal{N}$, and $n \in \N_0$ be given. Then the odd moments $M_{2n+1}^{\mu_{\Psi}}$ are zero and the even moments $M_{2n}^{\mu_{\Psi}}$ are given by
\[   
M_{2n}^{\mu_{\Psi}}(\varphi):=M_{2n}^{\mu_{\Psi}}(\varphi,\dots,\varphi)=\frac{1}{K}\frac{\prod_{i=1}^m \Gamma(b_i + \beta_i (n+1))}{\prod_{i=1}^p \Gamma(a_i + \alpha_i (n+1))} \frac{(2n)! }{n!2^{n}}  \langle\varphi,\varphi\rangle^{n}.
\]
In particular, for $\varphi,\psi \in \mathcal{N}$ we obtain
\begin{eqnarray}
\|\langle\cdot,\varphi\rangle\langle\cdot,\psi\rangle\|_{L^1(\mu_\Psi) }=M_{2}^{\mu_{\Psi}}(\varphi,\psi)&=&\frac{1}{K}\frac{\prod_{i=1}^m \Gamma(b_i + 2\beta_i) }{\prod_{i=1}^p \Gamma(a_i + 2\alpha_i) } \langle \varphi, \psi \rangle,\label{eq:mixed-moment2}\\
\|\langle\cdot,\varphi\rangle\|^2_{L^2(\mu_\Psi) }=M_{2}^{\mu_{\Psi}}(\varphi,\varphi)&=&\frac{1}{K}\frac{\prod_{i=1}^m \Gamma(b_i + 2\beta_i) }{\prod_{i=1}^p \Gamma(a_i + 2\alpha_i) }|\varphi|^2.\label{eq:mixed-moment2a}
\end{eqnarray}
\end{theorem}

\begin{remark}\label{rem:ext-dual-pairing}
Using the above theorem, it is possible to extend the dual pairing to $\mathcal{N}'\times\mathcal{H}$. More precisely, given $f\in\mathcal{H}$ there exists a sequence $(\varphi_n)_{n \in \N}\subset\mathcal{N}$ such that $\varphi_n\longrightarrow f$, $n\to\infty$ in $\mathcal{H}$.   It follows from \eqref{eq:mixed-moment2a} that $(\langle\cdot,\varphi_n\rangle)_{n\in\mathbb{N}}$ is a Cauchy sequence in $L^2({\mu_\Psi})$, hence it converges. Choosing a subsequence $(\langle \cdot, \varphi_{n_k} \rangle)_{k \in \N}$, we then define $\langle\cdot,f\rangle$ as an $L^2(\mu_{\Psi})$-limit of $\langle \cdot, \varphi_{n_k} \rangle$, that is
\begin{equation}
\langle\omega,f\rangle:=\lim_{k\to\infty}\langle\omega,\varphi_{n_k}\rangle,\quad \mu_\Psi\mbox{-a.a.}\;\omega\in\mathcal{N}'.
\end{equation}
\end{remark} 

The result of Remark~\ref{rem:ext-dual-pairing} allows us to define the characteristic functional of every element $f\in\mathcal{H}$. More precisely, we have
 $\lim_{k \to \infty} \exp(\mathrm{i}t\langle \omega, \varphi_{n_k} \rangle) = \exp(\mathrm{i}\lambda \langle \omega, f \rangle)$ for $\mu_{\Psi}$-a.a.~$\omega \in \mathcal{N}'$ and for all $\lambda \in \R$.
Noting that $|\exp(\mathrm{i}t \langle \omega, \varphi_{n_k} \rangle)|\le1$, we apply the dominated convergence theorem and obtain $\lambda\in\R$

\[  \int_{\mathcal{N}'} \e^{\mathrm{i}\lambda \langle \omega, f \rangle}\,\mathrm{d}\mu_{\Psi}(\omega)=\lim_{k\to \infty}\int_{\mathcal{N}'} \e^{\mathrm{i}\lambda \langle \omega, \varphi_{n_k} \rangle}\,\mathrm{d}\mu_{\Psi}(\omega)=\lim_{k\to \infty}\frac{1}{K}\, _m\!\Psi_p \left(-\frac{\lambda^2}{2} \langle \varphi_{n_k}, \varphi_{n_k} \rangle \right).
\]
By the convergence of $(\varphi_{n_k})_{k \in \N}$ and the continuity of $_m\!\Psi_p(\cdot)$, we obtain
\[ 
\int_{\mathcal{N}'} \e^{\mathrm{i}\lambda \langle \omega, f \rangle}\,\mathrm{d}\mu_{\Psi}(\omega) = \frac{1}{K}\, _m\!\Psi_p \left(-\frac{\lambda^2}{2} \langle f, f \rangle \right), \qquad \lambda \in \R.
\]
The Lemma \ref{lem:ProjeMoments} and Theorem \ref{thm:MomentmuHInfiniteDimension} can now be extended to all $f \in \mathcal{H}$. 
\section{Generalized Wright Analysis} \label{sec:AnalypropFHmeasu}
In this section, we prove the properties \eqref{PropertyA1} and \eqref{PropertyA2} for gWm. As a result, we obtain the existence of the Appell system, and a brief insight into the ergodicity and invariant property of $\mu_\Psi$ is given. 
Using the non-Gaussian analysis approach from \cite{KSWY98}, we construct the test function and distribution spaces, $(\mathcal{N})^{1}_{\mu_{\Psi}}$ and $(\mathcal{N})^{-1}_{\mu_{\Psi}}$, respectively. These spaces are characterized in terms of entire functions through the $S_{\mu_{\Psi}}$-transform. Finally, we establish two theorems characterizing the strong convergent of sequences and the integration of a family of distributions in $(\mathcal{N})^{-1}_{\mu_{\Psi}}$.

The Laplace transform of $\mu_\Psi\in\mathcal{M}_\Psi(\mathcal{N}')$ is defined by
\begin{equation}\label{def:LaplacemuC} l_{\mu_{\Psi}}(\vartheta):=\mathds{E}_{\mu_{\Psi}}\left(\e^{X_\vartheta}\right):=\int_{\mathcal{N}'}\exp(\langle\omega,\vartheta\rangle) \,\mathrm{d}\mu_{\Psi}(\omega) \in\C,  \quad \vartheta \in \mathcal{N}_{\C}. \end{equation}
\begin{itemize}
    \item The measure $\mu_\Psi$ has an analytic Laplace transform in a neighborhood of zero, i.e.~the mapping
\begin{equation}\label{PropertyA1}\tag{A1}
    \mathcal{N}_\C \ni \vartheta\mapsto l_{\mu_{\Psi}}(\vartheta)
\end{equation}
is holomorphic in a neighborhood $\mathcal{U}\subset\mathcal{N}_\C$ of zero.
\item For any nonempty open subset $\mathcal{U}\subset\mathcal{N}'$ it holds that \begin{equation} \label{PropertyA2} \tag{A2}
\mu_{\Psi}(\mathcal{U})> 0.
\end{equation}
\end{itemize}

Appell systems are considered for the measures $\mu_{\Psi}$ on $\mathcal{C}_{\sigma}(\mathcal{N}')$ that satisfy properties (A1) and (A2) below.

\subsection{Existence of Appell system}
In the sequel we show that gWm $\mu_{\Psi}$, satisfy \eqref{PropertyA1} and \eqref{PropertyA2}.

\begin{theorem}\label{thm:holomorphic-LT}
Let the assumptions of Lemma~\ref{lem:HFiniteDensities} together with $a^* \in (0,1)$ hold. For $\varphi \in \mathcal{N}$ and $\lambda \in \R$ the exponential function $\mathcal{N}' \ni \omega \mapsto \exp(|\lambda\langle \omega, \varphi \rangle|)$ is integrable and
\begin{equation}\label{eq:LaplacemuHInfiniteDimensionReal}  l_{\mu_\Psi}(\lambda\varphi)=\int_{\mathcal{N}'} \e^{\lambda \langle\omega,\varphi\rangle} \, \mathrm{d}\mu_{\Psi}(\omega)=\frac{1}{K}\, _m\!\Psi_p\left[ \genfrac{}{}{0pt}{}{(b_i+\beta_i,\beta_i)_{1,m}}{(a_j+\alpha_j,\alpha_j)_{1,p}} \middle| \frac{\lambda^2 \langle \varphi,\varphi \rangle }{2} \right]. 
\end{equation} 
Furthermore, for $\vartheta \in \mathcal{N}_{\C}$ and $z \in \C$, the map $l_{\mu_{\Psi}}(\cdot)$ is holomorphic from $\mathcal{N}_{\C}$ to $\C$:
\begin{equation} \label{eq:LaplacemuHInfiniteDimensionComplex} \int_{\mathcal{N}'} \exp(z\langle\omega,\vartheta\rangle) \, \mathrm{d}\mu_{\Psi}(\omega)=\frac{1}{K}\, _m\!\Psi_p\left[ \genfrac{}{}{0pt}{}{(b_i+\beta_i,\beta_i)_{1,m}}{(a_j+\alpha_j,\alpha_j)_{1,p}}\, \middle|\, \frac{z^2\langle \vartheta,\vartheta \rangle }{2} \right].
\end{equation}

\end{theorem}

\begin{proof}
The following proof follows the same ideas as in Lemma~4.1 and Theorem~4.2 in \cite{JahnI}. For completeness, we give the proof.\\
It is sufficient to show the integrability for $\lambda=1$. Thus, we show the integrability of $\exp(|\langle \omega, \varphi \rangle|)$ and $\exp(|\langle \omega, \vartheta \rangle|)$ for $\varphi \in \mathcal{N}$ and $\vartheta \in \mathcal{N}_\C$. Then, Equation~\eqref{eq:LaplacemuHInfiniteDimensionComplex} follows from Theorem~\ref{thm:MomentmuHInfiniteDimension} and the holomorphy of $l_{\mu_{\Psi}}(\vartheta)$ from the analiticity of $_m \Psi_p(\cdot)$; see Remark \ref{rem:Wright'sCharacteFunctionEntire}.
Let $\varphi \in \mathcal{N}$ be given. For every $N\in\mathbb{N}$ define the function $f_N(\cdot;\varphi)$ by
\[
\mathcal{N}'\ni \omega\mapsto f_N(\omega;\varphi):=\sum_{n=0}^N \frac{1}{n!}|\langle \omega, \varphi \rangle|^n\in\mathbb{R}.
\]
Clearly, we have
\[  \int_{\mathcal{N}'}|f_N(\omega;\varphi)| \, \mathrm{d}\mu_{\Psi}(\omega)<\sum_{n=0}^N \frac{1}{n!}\int_{\mathcal{N}'}|\langle \omega, \varphi \rangle|^n \, \mathrm{d}\mu_{\Psi}(\omega). \]
By separating the odd moments and even moments in the above sum, applying the Cauchy-Schwarz inequality and $ab \leq \frac{1}{2}(a^2+b^2)$ on the odd ones. Thus, we have 
\begin{eqnarray*}   \int_{\mathcal{N}'}|\exp(\langle \omega, \varphi \rangle)| \, \mathrm{d}\mu_{\Psi}(\omega)&\leq& \lim_{N \to \infty} \int_{\mathcal{N}'}f_N(\omega;\varphi) \, \mathrm{d}\mu_{\Psi}(\omega)\\
&\leq& \frac{3}{2K}\, _m\!\Psi_p\left(\frac{\langle \varphi, \varphi \rangle}{2}\right) +  \frac{1}{K}\,\, F(|\langle \varphi, \varphi \rangle|), 
\end{eqnarray*}
where
\[  
F(|\langle \varphi, \varphi \rangle|):= \frac{1}{2}\sum_{n\geq 0}\frac{\prod_{i=1}^m \Gamma(b_i +\beta_i+ \beta_i (n+1))}{\prod_{i=1}^p \Gamma(a_i + \alpha_i+ \alpha_i (n+1))} \frac{|\langle \varphi, \varphi \rangle|^{n+1}}{n!2^n}.
\] 
The last series converges because it is given by the gWf with other parameters; see Equation~\eqref{eq:PsiSeries}. 
Due to dominated convergence and Theorem~\ref{thm:MomentmuHInfiniteDimension} we have
\begin{eqnarray*}  \int_{\mathcal{N}'}\exp(\langle \omega, \varphi \rangle) \, \mathrm{d}\mu_{\Psi}(\omega)&=&\frac{1}{K} \lim_{N \to \infty}\sum_{n=0}^N \frac{\prod_{i=1}^m \Gamma(b_i + \beta_i (n+1))}{\prod_{i=1}^p \Gamma(a_i + \alpha_i (n+1))} \frac{\langle \varphi, \varphi \rangle^{n} }{n!2^{n}}  \\
&=&\frac{1}{K}\, _m\!\Psi_p\left(\frac{\langle \varphi, \varphi \rangle}{2}\right) .\end{eqnarray*}

For $\vartheta=\vartheta_1 + \mathrm{i}\vartheta_2 \in \mathcal{N}_\C$, we note that $|\exp(\langle \omega, \vartheta \rangle)|<\exp(|\langle \omega, \vartheta_1 \rangle|)$ and that $\exp(|\langle \omega, \vartheta_1 \rangle|) \in L^1(\mu_{\Psi})$. By considering that $\mathcal{N}$ (resp.~$\mathcal{N}_\C$) is a vector space over $\R$ (resp.~$\C$), Equations \eqref{eq:LaplacemuHInfiniteDimensionReal} and \eqref{eq:LaplacemuHInfiniteDimensionComplex} hold.
The function $l_{\mu_{\Psi}}: \mathcal{N}_{\C} \to \C$ is holomorphic on $\mathcal{N}_{\C}$ because of the composition of holomorphic functions.
\end{proof}

\begin{remark}
\begin{enumerate}
\item Assumption (A1) requires that $l_{\mu_{\Psi}} \in \mathrm{Hol}_0(\mathcal{N}_{\C})$ and the measures with this property are called analytic measures; see \cite{KSWY98}, page 221;
\item For a specific choice of $z \in \C$ and $\vartheta \in \mathcal{N}_\C$, the Laplace transform in Equation~\eqref{eq:LaplacemuHInfiniteDimensionReal} could be decomposed through the Fox-Hermite polynomials; see Lemma~\ref{lem:FoxHermitePolyPropert}-3.
\end{enumerate}
\end{remark}
A notable characteristic of the class $\mathcal{M}_\Psi(\mathcal{N}')$ of gWm is that they are a mixture of Gaussian measures with a probability measure on $(0,\infty)$; see Corollary~\ref{cor:mixture-fd} for the finite dimensional case. This property allows us to show that all measures $\mu_{\Psi}$ in the gWm class satisfy the property \eqref{PropertyA2}.

\begin{theorem}\label{thm:mixture}
Under the conditions of Lemma~\ref{lem:HFiniteDensities} every gWm $\mu_\Psi\in\mathcal{M}_\Psi(\mathcal{N}')$ is a mixture of the form
\begin{equation}\label{eq:mixture}	\mu_{\Psi}=\int_0^\infty\mu^{(s)}\,\varrho(s)\,\mathrm{d}s, 
\end{equation}
where $\varrho(\cdot)$ is the probability density on $(0,\infty)$  given in Equation \eqref{eq:densitiesFiniteMoments} and  $\mu^{(s)}$ denote the centered Gaussian measures on $\mathcal{N}'$ with variance $s\geq 0$, i.e.
\[ \int_{\mathcal{N}'}\exp(\mathrm{i}\langle\omega,\xi\rangle)\,\mathrm{d}\mu^{(s)}(\omega)=\exp\left(-\frac{s}{2}\langle\xi,\xi\rangle\right),\quad\xi\in \mathcal{N}. \]
Moreover, $\mu_\Psi$ satisfies the property \eqref{PropertyA2}.
\end{theorem}

\begin{proof} 
The equality \eqref{eq:mixture} follows by taking the Fourier transform on both sides, considering that $_m\!\Psi_p$ is the Laplace transform of the density $\varrho$. The property \eqref{PropertyA2} follows from \eqref{eq:mixture} and the fact that $\mu^s(U)>0$ for each $s\in(0,\infty)$ and for every nonempty set $U\subset\mathcal{N}'$. 
\end{proof}
It is a general fact that the mixture of Gaussian measures using a bounded Borel measure is invariant under the group of linear automorphisms on $\mathcal{N}'$ (denoted by $\mathrm{Aut}(\mathcal{N}')$); see Ch.~5 in \cite{Hida80} for more details. Since our class $\mathcal{M}_\Psi(\mathcal{N}')$ of measures are mixtures of Gaussian measures, they share this property. We state these results in the next corollary.

\begin{corollary}
Consider the class $\mathcal{M}_\Psi(\mathcal{N}')$ of gWm from Definition~\ref{def:FoxHmeasu}. Then, for every $\mu_\Psi\in\mathcal{M}_\Psi(\mathcal{N}')$
\begin{enumerate}
    \item $\mu_{\Psi}$ is invariant under $\mathrm{Aut}(\mathcal{N}')$,
    \item $\mu_{\Psi}$ is not ergodic under $\mathrm{Aut}(\mathcal{N}')$.
\end{enumerate}
\end{corollary}

\begin{proof}
    The result follows using the representation in Equation~\eqref{eq:mixture}, Theorem~5.6, and Proposition~5.6 in \cite{Hida80}.
\end{proof}

\subsection{\texorpdfstring{$S_{\mu_{\Psi}}$}{Lg}, \texorpdfstring{$T_{\mu_{\Psi}}$}{Lg}-transforms and characterization theorems}
\label{subsec:charact-thms}

Under the assumptions of Lemma~\ref{lem:HFiniteDensities} together with $a^* \in (0,1)$, we can build the following infinite dimensional setting to study functionals and distributions. \\
Sec.~5 and Sec.~6 in \cite{KSWY98}, the test function space, i.e. $(\mathcal{N})^{1}$, and the distribution space, i.e.~$(\mathcal{N})^{-1}_{\mu_{\Psi}}$, exist, and we have:
\[  (\mathcal{N})^{1}\subset L^2(\mu_{\Psi}) \subset(\mathcal{N})^{-1}_{\mu_{\Psi}}\]
endowed with the dual pairing $\langle\!\langle \cdot,\cdot \rangle\!\rangle_{\mu_{\Psi}}$ between $ (\mathcal{N})^{-1}_{\mu_{\Psi}} $ and $(\mathcal{N})^{1}$ which is the bilinear extension of the inner product of $L^2(\mu_{\Psi})$. \\
\begin{remark}
    In \cite{KSWY98}, p.242, a finer chain of spaces is built to characterize the singularity of distributions.
\end{remark}
We define the $S_{\mu_{\Psi}}$-transform by means of the normalized exponential $e_{\mu_{\Psi}}(\cdot, \xi)$:

\[ (S_{\mu_{\Psi}}\Phi )(\xi):= \langle\!\langle \Phi, e_{\mu}(\xi,\cdot) \rangle\!\rangle_{\mu_{\Psi}}=\frac{1}{l_{\mu_{\Psi}}(\xi)}\int_{\mathcal{S}'}\e^{\langle \omega, \xi\rangle}\Phi(\omega)\mu_{\Psi}(d\omega), \quad \xi \in  U_{l,k}. \]
where 

\[ l_{\mu_{\Psi}}(\lambda\xi)=\frac{1}{K}\, _m\!\Psi_p\left[ \genfrac{}{}{0pt}{}{(b_i+\beta_i,\beta_i)_{1,m}}{(a_j+\alpha_j,\alpha_j)_{1,p}} \middle| \frac{\lambda^2 \langle \xi,\xi \rangle }{2} \right]\]
 and $U_{l,k}=\{\xi \in \mathcal{N}_{\mathbb{C}}\mid \, 2^k |\xi|_l <1\}$. We note that the normalized exponential $e_{\mu_{\Psi}}(\cdot,\xi)$ is a test function of finite order; see Example 6 and Section~7 in \cite{KSWY98}

The properties \eqref{PropertyA1} and \eqref{PropertyA2} allow us to state the following theorem, which is a special case of Theorem 8.34 in \cite{KSWY98}. It
characterizes the elements of $(\mathcal{N})^{-1}_{\mu_\Psi}$ in terms of holomorphic functions via the $S_{\mu_\Psi}$-transform.
\begin{theorem}\label{thm:StransIsomor}
The $S_{\mu_{\Psi}}$-transform is a topological isomorphism from $(\mathcal{N})^{-1}_{\mu_{\Psi}}$ to $\mathrm{Hol}_{0}(\mathcal{N}_{\mathbb{C}})$.
\end{theorem}

The above characterization theorem leads directly to two corollaries for integrals of elements in $(\mathcal{N})^{-1}_{\mu_{\Psi}}$ in a weak sense and the convergence of sequences in $(\mathcal{N})^{-1}_{\mu_{\Psi}}$.
\begin{theorem}[see Thm.~4.10~in \cite{JahnI} for the Mittag-Leffler measure]\label{thm:CharIntegrableMap}
Let $(T,\mathcal{B},\nu)$ be a measure space and $\Phi_t\in (\mathcal{N})^{-1}_{\mu_{\Psi}}$ for all $t\in T$. Let $\mathcal{U}\subset\mathcal{N}_\C$ be an appropriate neighborhood of zero and $C$ a positive constant such that:
\begin{enumerate}
	\item $S_{\mu_{\Psi}}\Phi_\cdot(\xi)\colon T\to\C$ is measurable for all $\xi\in\mathcal{U}$.
	\item $\int_T |S_{\mu_{\Psi}}\Phi_t(\xi)|\,\mathrm{d}\nu(t) \leq C$ for all $\xi\in\mathcal{U}$.
\end{enumerate}
Then there exists $\Xi \in (\mathcal{N})^{-1}_{\mu_{\Psi}}$ such that for all $\xi\in\mathcal{U}$
\[
	S_{\mu_{\Psi}}\Xi(\xi) = \int_T S_{\mu_{\Psi}}\Phi_t(\xi)\,\mathrm{d}\nu(t).
\]
We denote $\Xi$ by $\int_T \Phi_t\,\mathrm{d}\nu(t)$ and call it the weak integral of $\Phi$.
\end{theorem}

\begin{definition}
For $\Phi \in (\mathcal{N})^{-1}_{\mu_{\Psi}}$ and $\xi \in U_{l,k}=\{ \xi \in \mathcal{N}_{\mathbb{C}} |\; 2^k |\xi|_{l}<1 \}$, we define the $T_{\mu_{\Psi}}$-transform by
\[  
(T_{\mu_{\Psi}}\Phi)(\xi)=\langle\! \langle \Phi, \e^{\mathrm{i}\langle \cdot, \xi \rangle } \rangle \!\rangle_{\mu_{\Psi}}=\int_{\mathcal{S}'}\e^{\mathrm{i}\langle \omega, \xi\rangle}\Phi(\omega)\, \mu_{\Psi}(\mathrm{d}\omega). 
\]
\end{definition}
\begin{remark}
We recall that if $\xi=0$, then $\exp{(\mathrm{i} \langle \cdot, \xi \rangle)}=1$ and we have 
\[ (T_{\mu_{\Psi}}\Phi)(0)=\mathbb{E}_{\mu_{\Psi}}(\Phi). \]
\end{remark}

\section{Applications}
\label{sec:DonskDelta}
\noindent Donsker's delta is an important example of a Hida distribution within Gaussian analysis, used extensively across quantum field theory, stochastic differential equations, and mathematical finance. For further details, refer to \cite{HKPS93, Wes95, AOU01} and the references therein. In this section, we introduce a distribution in $(\mathcal{N})^{-1}_{\mu_{\Psi}}$ that is the analog of Donsker's delta in Gaussian analysis, i.e.~$\delta_a(\cdot)$, $a \in \R$. More precisely, we want to give meaning to the composition $\delta_a(\langle \cdot, \eta \rangle)$, $\eta \in \mathcal{H}$, at $a=0$.

We proceed as follows:

\begin{enumerate}

\item Use the integral representation for the Dirac delta distribution $\delta$ and give sense to the expression
\begin{equation}\label{eq:deltadistrib}
	\delta(\langle\cdot,\eta\rangle) = \frac{1}{2\pi} \int_\R \mathrm{e}^{\mathrm{i}x\langle\cdot,\eta\rangle}\,\mathrm{d} x,\quad\eta\in\mathcal{H}.
\end{equation}
Note that $\exp(\mathrm{i}x\langle\cdot,\eta\rangle)\in L^2(\mu_{\Psi})$ because $|\exp(\mathrm{i}x\langle\omega,\eta\rangle)|^2=1$ and we have a probability measure.

\item We show that the right hand side of \eqref{eq:deltadistrib} defines an element in $(\mathcal{N})^{-1}_{\mu_{\Psi}}$ by Theorem~\ref{thm:CharIntegrableMap}. The $T_{\mu_{\Psi}}$-transform of the integrand $\exp(\mathrm{i}x\langle\cdot,\eta\rangle)$ can be computed as follows.
For every $\varphi\in\mathcal{N}_\C$ and $\eta \in \mathcal{N}$ it follows from Equation \eqref{eq:FoxHmeasure} that
\begin{align}
	T_{\mu_{\Psi}}\exp(\mathrm{i}x\langle\cdot,\eta\rangle)(\varphi) &= \int_{\mathcal{N}'} \exp(\mathrm{i}\langle\omega,x\eta+\varphi\rangle)\,\mathrm{d}\mu_{\Psi}(\omega) \nonumber \\
	&= \frac{1}{K} \, _m\!\Psi_q\left(-\frac{1}{2}x^2\langle\eta,\eta\rangle -\frac{1}{2} \langle\varphi,\varphi\rangle - x\langle\varphi,\eta\rangle \right).\label{eq:Ttransform}
\end{align}
We note that $x\mapsto T_{\mu_{\Psi}}\exp(\mathrm{i}x\langle\cdot,\eta\rangle)(\varphi)$ is measurable because it is a composition of a completely monotone function and a polynomial. By Corollary~3.8 in \cite{BCDS23} and Theorem~2.9 in \cite{SaiKil}, we extend the function $_m\Psi_p(\cdot)$ to the complex plane for $a^*\in(0,1)$.
\item We show that $\int_\R |T_{\mu_{\Psi}}\exp(\mathrm{i} x \langle\cdot,\eta\rangle)(\varphi)|\,\mathrm{d} x$ is bounded for all $\varphi$ in a neighborhood of zero of $\mathcal{N}_\C$. We assume in addition that
\begin{equation}\label{eq:HypParamDonsker}
 \quad 2b_j+\beta_j>0, \quad j=1,\dots,m. \end{equation}

\item We summarize the results in Theorem \ref{theorem:Donsker}.

\end{enumerate}
\begin{proposition}\label{Prop:Integral}
Let the assumptions of Lemma \ref{lem:HFiniteDensities} together with $a^* \in (0,1)$ and $2b_j + \beta_j>0$,  $j=1,\dots,m$, hold. 
For $\eta\in\mathcal{H}$, $\varphi\in\mathcal{N}_\C$ and $x\in\R$ let $z(x,\eta,\varphi):=\tfrac{1}{2} x^2\langle\eta,\eta\rangle + \tfrac{1}{2}\langle\varphi,\varphi\rangle + x\langle\eta,\varphi\rangle$. Then there exists a constant $C>0$ such that
\[
\int_\R |\,_m\!\Psi_p(-z(x,\eta,\varphi))|\,\mathrm{d} x \leq C,\quad\varphi\in\mathcal{U}_M := \big\{ \varphi\in\mathcal{N}_\C \mid |\varphi| < M \big\}
\]
for any $M>0$.
\end{proposition}

\begin{proof}
Let $\eta \in \mathcal{H}$, $\varphi \in\mathcal{U}_M$ with $M>0$ be given. First, we note that $\,_m\!\Psi_p(-z(x,\eta,\varphi))$ is measurable. Using the fact that $_m\!\Psi_p(\cdot)$ is the Laplace transform of the density $\varrho$ (cf.~Equation~\eqref{eq:LTdensitiesFiniteMoments}) and using Fubini's theorem we obtain
\begin{align*}
	&\int_\R |\,_m\!\Psi_p(-z(x,\eta,\varphi))|\,\mathrm{d} x	\leq \int_0^\infty H^{m,0}_{p,m}(r) \int_\R \exp\left(-r\Re(z(x,\eta,\varphi))\right)\,\mathrm{d} x\,\mathrm{d} r \\
	&=\sqrt{\frac{2\pi}{\langle\eta,\eta\rangle}} \int_0^\infty H^{m,0}_{p,m}(r) r^{-1/2} \exp\left(\frac{r}{2}\left(\frac{\langle\eta,\varphi_1\rangle^2}{\langle\eta,\eta\rangle}-|\varphi_1|^2+|\varphi_2|^2 \right) \right) \,\mathrm{d} r.
\end{align*}
For $\varphi \in \mathcal{U}_M$, we have by Cauchy-Schwarz inequality that
\begin{align*}
	\frac{\langle\eta,\varphi_1\rangle^2}{\langle\eta,\eta\rangle}+|\varphi_2|^2-|\varphi_1|^2<M^2.
\end{align*}
This yields
\[
	\int_\R |\,_m\!\Psi_p(-z(x,\eta,\varphi))|\,\mathrm{d} x \leq \sqrt{\frac{2\pi}{\langle\eta,\eta\rangle}} \int_0^\infty H^{m,0}_{p,m}(r) r^{-1/2} \exp\left(\frac{1}{2}M^2 r \right) \,\mathrm{d} r,
\]
where the integral on the right side is finite by applying Equation (2.8.12) in \cite{SaiKil} with
\[
H^{M,N}_{P,Q}\left[z \, \Bigg| \genfrac{}{}{0pt}{}{(c_i,\gamma_i)_P}{(d_j,\delta_J)_Q}\right]=H^{1,0}_{0,1}\left[-\frac{M^2}{2}r \, \Bigg| \genfrac{}{}{0pt}{}{-\!-}{(0,1)}\right]=\exp\left(\frac{M^2}{2}r\right).\qedhere
\]
\end{proof}

\begin{theorem}\label{theorem:Donsker}
Let $0\neq\eta\in\mathcal{H}$ be given. Then Donsker's delta is defined via the integral
\[
	\delta(\langle\cdot,\eta\rangle) := \frac{1}{2\pi}\int_\R \exp\left( \mathrm{i}x\langle\cdot,\eta\rangle \right)\,\mathrm{d} x,
\]
and exists in the space $\left(\mathcal{N}\right)^{-1}_{\mu_{\Psi}}$ as a weak integral in the sense of Theorem \ref{thm:CharIntegrableMap}. Moreover for all $\varphi\in\mathcal{U}_M$, $M>0$, as in Proposition \ref{Prop:Integral}, we have
\[
	\left( T_{\mu_{\Psi}}\delta(\langle\cdot,\eta\rangle)\right)(\varphi) = \frac{1}{K\sqrt{2\pi\langle\eta,\eta\rangle}} \, _m\!\Psi_q \left( \genfrac{}{}{0pt}{}{(b_j+\beta_j/2,\beta_j)_m}{(a_i+\alpha_i/2,\alpha_i)_p} \Bigg| -\frac{1}{2}\left( \langle\varphi,\varphi\rangle - \frac{\langle\eta,\varphi\rangle^2}{\langle\eta,\eta\rangle}\right)\right). 
\]
\end{theorem}

\begin{proof}
Due to Proposition \ref{Prop:Integral} there exists a constant $C>0$ such that
\[
	\frac{1}{2\pi} \int_\R \left( T_{\mu_{\Psi}}\exp(\mathrm{i}x\langle\cdot,\eta\rangle)\right)(\varphi)\,\mathrm{d} x < C,\quad\varphi\in\mathcal{U}_M.
\]
Thus, by Theorem \ref{thm:CharIntegrableMap} it follows that $\delta(\langle\cdot,\eta\rangle)$ exists as an element in $(\mathcal{N})^{-1}_{\mu_{\Psi}}$. Finally we calculate the $T_{\mu_{\Psi}}$-transform of Donsker's delta using Theorem~2.9 from \cite{SaiKil} in $*$ below:
\begin{align*}
	&T_{\mu_{\Psi}}\delta(\langle\cdot,\eta\rangle)(\varphi) = \frac{1}{2\pi K} \int_\R \,_m\!\Psi_p(-z(x,\eta,\varphi))\,\mathrm{d} x \\
	&= \frac{1}{2\pi K} \int_0^\infty H^{m,0}_{p,m}(r) \exp\left( -\frac{1}{2} r\langle\varphi,\varphi\rangle\right) \int_\R \exp\left( -\frac{1}{2} r\langle\eta,\eta\rangle x^2 - r\langle\eta,\varphi\rangle x\right)\mathrm{d} x\,\mathrm{d} r \\
	&= \frac{1}{K\sqrt{2\pi\langle\eta,\eta\rangle}} \int_0^\infty H^{m,0}_{p,m}(r) r^{-1/2} \exp\left(-\frac{r}{2}\left( \langle\varphi,\varphi\rangle - \frac{\langle\eta,\varphi\rangle^2}{\langle\eta,\eta\rangle}\right)\right)\mathrm{d} r\\
 &\overset{*}{=} \frac{1}{K\sqrt{2\pi\langle\eta,\eta\rangle}} H^{1,m}_{m,p+1}\left[ \frac{1}{2}\left( \langle\varphi,\varphi\rangle - \frac{\langle\eta,\varphi\rangle^2}{\langle\eta,\eta\rangle}\right) \, \Bigg| \genfrac{}{}{0pt}{}{(1-b_j-\beta_j/2,\beta_j)_m}{(0,1),(1-a_{i-1}-\alpha_{i-1}/2,\alpha_{j-1})_{2,p+1}}\right].
\end{align*}
Applying Equation \eqref{eq:defGenWrightfuncasFoxHfunc}, we conclude the proof.
\end{proof}

\begin{remark}\label{rem:TtransfDonskerSeriesExpansion}
We can express the $T_{\mu_{\Psi}}$-transform of Donsker's delta as a series expansion, which turns out to be a known special function for the proper choice of the parameters.
\begin{enumerate}
\item We find the series expansion of the $T_{\mu_{\Psi}}$-transform of Donsker's delta with the help of Remark \ref{rem:Wright'sCharacteFunctionEntire}:
\begin{align*}
	&H^{1,m}_{m,p+1}\left[z \, \Bigg| \genfrac{}{}{0pt}{}{(1-b_j-\beta_j/2,\beta_j)_m}{(0,1),(1-a_{i-1}-\alpha_{i-1}/2,\alpha_{j-1})_{2,p+1}}\right]\\
    &= \, _m\!\Psi_q \left( \genfrac{}{}{0pt}{}{(b_j+\beta_j/2,\beta_j)_m}{(a_i+\alpha_i/2,\alpha_i)_p} \Bigg| \, -z\right) \\
    &= \sum_{k=0}^\infty \frac{ \prod_{j=1}^m \Gamma(b_j+\beta_j(1/2+k))}{ \prod_{i=1}^p \Gamma(a_i+\alpha_i(1/2+k))}\frac{(-z)^k}{k!},
\end{align*}
for $z\in\C$.
\item The $T_{\mu_{\Psi}}$-transform of Donsker's delta could be decomposed into a series of polynomials by using Fox-Hermite polynomials; see Lemma~\ref{lem:FoxHermitePolyPropert}-3.
\item For $b_1=0,\, \beta_1=1,\, a_1=1-\rho,\, \alpha_1=\rho$ the $T_{\mu_{\Psi}}$-transform of Donsker's delta in generalized Wright analysis coincides with $T_{\mu_\rho}$-transform of Donsker's delta in Mittag-Leffler analysis; see \cite{JahnI}. 

\item Note that in the case $\rho=1$:
\[
	\sum_{k=0}^\infty \frac{(-1)^k \Gamma(k+1/2)}{k! \Gamma(1+k-1/2)}z^k = \mathrm{e}^{-z}.
\]
\end{enumerate}
\end{remark}

\begin{corollary}
The generalized expectation of Donsker's delta is given by 
\[  
	\mathbb{E}_{\mu_\Psi}\left(\delta\left(\langle\cdot,\eta\rangle\right)\right) = \left( T_{\mu_{\Psi}} \delta\left(\langle\cdot,\eta\rangle\right)\right)(0) =  \langle\!\langle\delta(\langle\cdot,\eta\rangle),1\rangle\!\rangle_{\mu_{\Psi}}.
\]
More specifically, we get:
\[
	\mathbb{E}_{\mu_\Psi}\left(\delta\left(\langle\cdot,\eta\rangle\right)\right)  = \frac{1}{K\sqrt{2\pi \langle\eta,\eta\rangle}}\frac{ \prod_{j=1}^m \Gamma(b_j+\beta_j/2)}{ \prod_{i=1}^p \Gamma(a_i+\alpha_i/2)}.
\]
\end{corollary}

\begin{proof}
    Evaluating the series expansion from Remark \ref{rem:TtransfDonskerSeriesExpansion}-1) at $z=0$, the result follows.  
\end{proof}

In the same way, we can define Donsker's delta at any arbitrary point $a\in\R\backslash\{0\}$:
\begin{proposition}\label{prop:Donskersdeltaina}
Let $0\neq \eta\in\mathcal{H}$ be given and assume that \eqref{eq:HypParamDonsker} holds. Then
\[
	\delta_a(\langle\cdot,\eta\rangle) = \frac{1}{2\pi}\int_\R \exp\big( \mathrm{i} x(\langle\cdot,\eta\rangle-a)\big)\,\mathrm{d} x
\]
exists in $\left(\mathcal{N}\right)^{-1}_{\mu_{\Psi}}$ as a weak integral in the sense of Theorem \ref{thm:CharIntegrableMap} and defines Donsker's delta in $a\in\R\backslash\{0\}$.
\end{proposition}
\begin{proof}
The $T_{\mu_{\Psi}}$-transform of the integrand for $\varphi\in\mathcal{N}_\C$ is given by:
\[
	\frac{1}{2\pi K}\exp(-\mathrm{i}xa) \,_m\!\Psi_p\left( -\frac{1}{2} x^2\langle\eta,\eta\rangle - \frac{1}{2}\langle\varphi,\varphi\rangle - x\langle\varphi,\eta\rangle \right) .
\]
Hence its absolute value coincides with the $T_{\mu_{\Psi}}$-transform 
in the case $a=0$. Now we can proceed as in the case $a=0$:
\begin{align*}
    &\left( T_{\mu_{\Psi}} \delta_a\left(\langle\cdot,\eta\rangle\right)\right)(\varphi)\\
    &= \frac{1}{2 \pi K} \int_\R \e^{-\mathrm{i}xa} \,_m\!\Psi_p\left( -\frac{1}{2} x^2\langle\eta,\eta\rangle - \frac{1}{2}\langle\varphi,\varphi\rangle - x\langle\varphi,\eta\rangle \right) \mathrm{d}x \\
    &\overset{*}{=}\frac{1}{2 \pi K} \int_0^\infty H^{m,0}_{p,m}(r) \exp\left(-\frac{1}{2} r \langle \varphi, \varphi \rangle\right)\int_\R \exp\left(-\frac{1}{2}rx^2 \langle \eta,\eta \rangle-x(r \langle \varphi, \eta \rangle +\mathrm{i}a)\right)\mathrm{d}x\,\mathrm{d}r\\
    &\overset{**}{=} \frac{\exp( \mathrm{i}a \langle \varphi, \eta \rangle / \langle \eta,\eta \rangle)}{K \sqrt{2 \pi \langle \eta,\eta \rangle}} \int_0^\infty H^{m,0}_{p,m}(r)r^{-1/2} \exp\left( -\frac{r}{2}\left( \langle \varphi, \varphi \rangle - \frac{\langle \varphi, \eta \rangle^2}{\langle \eta,\eta \rangle} \right) - \frac{a^2}{2 r \langle \eta, \eta \rangle} \right) \mathrm{d}r.
\end{align*}
where in $*$ we used the integral representation of $_m \Psi_p (\cdot)$ and in $**$ we use the following Gaussian integral with $r \in (0,\infty)$ 
\[ \int_\R \exp\left(-\frac{1}{2}rx^2 \langle \eta,\eta \rangle-x(r \langle \varphi, \eta \rangle +\mathrm{i}a)\right)\mathrm{d}x= \sqrt{\frac{2 \pi}{r \langle \eta, \eta \rangle}} \exp\left( \frac{(r \langle \varphi, \eta \rangle +\mathrm{i}a)^2}{2 r \langle \eta, \eta \rangle} \right).  \]
Thus, we can read this mean using a Gaussian random variable $X$, as follows:
\begin{eqnarray*} 
\mathbb{E}_{\mu_{\Psi}}(\delta_a(\langle \cdot, \eta \rangle))&=&\left( T_{\mu_{\Psi}} \delta_a\left(\langle\cdot,\eta\rangle\right)\right)(0)\\ 
&=&\frac{1}{K \sqrt{2 \pi \langle \eta,\eta \rangle}} \int_0^\infty H^{m,0}_{p,m}(r) r^{-1/2} \exp\left(-\frac{a^2}{2 r \langle \eta, \eta \rangle}\right)\mathrm{d}r\\
&=& \mathbb{E}_{\mu_{R,\eta}}(\delta_a(X)),   
\end{eqnarray*}
where $\mu_{R,\eta}$ is the distribution of the centered Gaussian random variable $X$ with variance $R \langle \eta, \eta \rangle$ and $R \in \mathcal{X}$; see Lemma~\ref{lem:HFiniteDensities}. Furthermore, we use Equation~(2.8.12) in \cite{SaiKil} to represent this mean through the density of $\mu_{\Psi}^1$:

\begin{eqnarray*}
    &&\frac{1}{K \sqrt{2 \pi \langle \eta,\eta \rangle}} \int_0^\infty H^{m,0}_{p,m}(r) r^{-1/2} \exp\left(-\frac{a^2}{2 r \langle \eta, \eta \rangle}\right)\mathrm{d}r\\
    &=&\frac{1}{\sqrt{2 \pi \langle \eta, \eta \rangle}K} H^{m+1,0}_{p,m+1}\left[ \frac{|a|^2}{2 \langle \eta, \eta \rangle}\,\bigg| \genfrac{}{}{0pt}{}{(a_i+\alpha_i/2,\alpha_i)_p}{(0,1),(b_{j-1}+\beta_{j-1}/2,\beta_{j-1})_{2,m+1}}\right],\end{eqnarray*} 
see also Theorem \ref{thm:FiniteDimensionalFoxHGaussianDensity}.
\end{proof}

\section{Conclusions}
We summarize the results and provide a brief overview of the forthcoming studies.

We use gWfs, $_m\Psi_q(\cdot)$, to define a class of measures $\mu_\Psi$ in finite and infinite dimensions. We investigate some properties of this class of measures, e.g., mixed moments, density, analyticity, invariance, etc. We have established the generalized Wright analysis by showing the existence of an Appell system, which allows us to build a chain of tests and generalized function spaces.  We studied Donsker's delta in this non-Gaussian setting as a well-defined element in a certain distribution space. We will postpone the stochastic counterpart of this framework for a future paper. Choosing a specific nuclear triple, we may define a generalized stochastic process that can be realized in different forms using known processes; see \cite{BCM23} and \cite{JahnII} for similar examples. Specifically, it is expected that these processes behave like anomalous diffusions and their fractional Fokker-Planck equations. Moreover, we plan to study the Green function corresponding to the time-fractional heat equation that corresponds to an extension of the Feynman-Kac formula in this non-Gaussian setting. Another application we have in mind is to study the local times and the self-intersection local times corresponding to the associated process using Theorem~\ref{theorem:Donsker} and Proposition~\ref{prop:Donskersdeltaina}.

\appendix
\section{Appendix}
\label{app:Completely monotone FHDAM}
Certainly, the class of F$H$dam represents a wide class of densities. On the other hand, it is natural to ask if the latter have important properties such as infinite divisibility. Here we give a partial answer to this question.\\ 
It is well known that the complete monotonicity and infinite divisibility are strictly related; see Theorem~51.6 in \cite{Sat99}. 
In the following lemma, we highlight why an F$H$dam, with the hypotheses given in the above lemma, is not completely monotone. In fact, such an F$H$dam cannot exist given that $a^* \in (0,1)$ and $\mathcal{L}=\mathcal{L}_{\mathrm{i}\gamma \infty}$.
\begin{lemma}
    Let $m,p \in \N$, $a_i,b_j \in \R$ and $\alpha_i,\beta_j>0$ for $i=1,\dots, p$ and $j=1,\dots,m$. Let $H^{m,0}_{p,m}(x), x>0,$ be a density with $a^*\in (0,1)$, then $H^{m,0}_{p,m}(x), x>0,$ is not completely monotone.
\end{lemma}
\begin{proof}
    In the following we prove the result using an argument to absurdity: we suppose to have a F$H$dam density with $\mathcal{L}=\mathcal{L}_{\mathrm{i} \gamma \infty}$ and parameter $a^*$, denoted by $a^*_{H^{m,0}_{p,m}}$, is in $(0,1)$ and we also assume that the same F$H$dam is completely monotone, then we found that the integrand function in Equation~\eqref{def:KernelFunctionFoxH} of the Fox-$H$ density, whose Laplace transform coincides with the chosen F$H$dam, diverge on $\mathcal{L}=\mathcal{L}_{\mathrm{i} \gamma \infty}$.\\
    Let us assume that $H^{m,0}_{p,m}(x),x>0,$ is completely monotone, then by Bernstein theorem and Theorem 2.3 in \cite{SaiKil} we have that it is the Laplace transform of a Fox-$H$ density. If one of the couple $(b_{j
    _0},\beta_{j_0})=(0,1)$, we have, with $x>0$,
    
    \[H^{m,0}_{p,m} \left[ x \,\Bigg|\genfrac{}{}{0pt}{}{(a_i,\alpha_i)_{1,p}}{(0,1),(b_j,\beta_j)_{\{1,m\} \backslash \{j_0\}}}\right]=\left( \mathscr{L}H^{0,m-1}_{m-1,p} \left[ \cdot \,\Bigg|\genfrac{}{}{0pt}{}{(1-(b_j +\beta_j),\beta_j)_{1,m-1}}{(1-(a_i+\alpha_i),\alpha_i)_{1,p}} \right]\right)(x),\]
    if there is no couple equal to $(0,1)$ we could apply Equation (2.1.2) in \cite{SaiKil} such that
    \begin{eqnarray*}
    H^{m,0}_{p,m} \left[ x \,\Bigg|\genfrac{}{}{0pt}{}{(a_i,\alpha_i)_{1,p}}{(b_j,\beta_j)_{1,m}}\right]&=&H^{m+1,0}_{p+1,m+1} \left[ x \,\Bigg|\genfrac{}{}{0pt}{}{(a_i,\alpha_i)_{1,p},(0,1)}{(0,1),(b_j,\beta_j)_{\{1,m\}}}\right] \\
    &=&\left( \mathscr{L}H^{0,m}_{m,p+1} \left[ \cdot \,\Bigg|\genfrac{}{}{0pt}{}{(1-(b_j+\beta_j),\beta_j)_{1,m}}{(1-(a_i+\alpha_i),\alpha_i)_{1,p}(0,1)} \right]\right)(x),
    \end{eqnarray*}
    where $(\mathscr{L}\varrho(\cdot))(x)$ denotes the Laplace transform of $\varrho(\cdot)$.\\
    For the first case we proceed as follow. If we denote by $a^*_{H^{0,m-1}_{m-1,p}}, \Delta_{H^{0,m-1}_{m-1,p}}, a^*_{1,H^{0,m-1}_{m-1,p}}$ and $a^*_{2,H^{0,m-1}_{m-1,p}}$ the parameters, referred to $H^{0,m-1}_{m-1,p}$, defined, respectively, in Equations~(1.1.7), (1.1.8), (1.1.11) and (1.1.12) in \cite{SaiKil}, we have the following chain $0<1-a^*_{H^{0,m-1}_{m,p-1}}=1-\Delta_{H^{0,m}_{m,p}}=\Delta_{H^{0,m-1}_{m-1,p}}=-a^*_{H^{0,m-1}_{m-1,p}}=-a^*_{2,H^{0,m-1}_{m-1,p}}$ and $a^*_{1,H^{0,m-1}_{m-1,p}}=0$. Hence, we have that, for any $x>0$, the integrand function $|\mathcal{H}^{m,0}_{p,m}(s)x^{s}|, s=\gamma \pm \mathrm{i}R \in \mathcal{L}_{\mathrm{i}\gamma \infty},$ diverges as $R \to \infty$. In fact, the leading term of its asymptotic behaviour is $\exp(\pi a^*_{2,H^{0,m-1}_{m-1,p}}R/2)$; see Equation~(3) in \cite{BCDS23}. Then, the Fox-$H$ function $H^{0,m-1}_{m-1,p}\left[ \cdot \,\Big|\genfrac{}{}{0pt}{}{(1-(b_j +\beta_j),\beta_j)_{1,m-1}}{(1-(a_i+\alpha_i),\alpha_i)_{1,p}} \right](x)$ cannot exist for any $x>0$.\\
    The same argument holds to prove that $H^{0,m}_{m,p+1} \left[ \cdot \,\Big|\genfrac{}{}{0pt}{}{(1-(b_j+\beta_j),\beta_j)_{1,m}}{(1-(a_i+\alpha_i),\alpha_i)_{1,p}(0,1)} \right]$ cannot exist for any $x>0$.
\end{proof}
\bibliographystyle{plain}
\bibliography{Hanalysis}

\end{document}